\documentclass[11pt, reqno]{amsart}
\usepackage[dvips]{graphicx}
\usepackage{bmpsize} 
\usepackage[margin=1.1in]{geometry}

\usepackage{mathrsfs}
\usepackage{amsmath,amscd}
\usepackage{amssymb}
\usepackage{amscd}
\graphicspath{/home/jhois/Documents/Images/}  

\newtheorem{theorem}{Theorem}[section]
\newtheorem{proposition}[theorem]{Proposition}
\newtheorem{lemma}[theorem]{Lemma}
\newtheorem{corollary}[theorem]{Corollary}

\theoremstyle{definition} 
\newtheorem{definition}[theorem]{Definition}
\newtheorem{example}[theorem]{Example}
\newtheorem{remark}[theorem]{Remark}

\newcommand{\C}{\mathbb C} 
\newcommand{\R}{\mathbb R} 
\newcommand{\Z}{\mathbb Z}

\DeclareMathOperator{\arccot}{arccot}

\numberwithin{equation}{section} 
\numberwithin{theorem}{section}
\numberwithin{figure}{section}

\begin{document}

\author[J.A.~Hoisington]{Joseph Ansel Hoisington} \address{Department of Mathematics, University of Pennsylvania, Philadelphia, PA 19104-6395 USA}
\address{(Current Address: Department of Mathematics, University of Georgia, Athens, GA 30602 USA)}\email{jhoisington@uga.edu}

\title[The Total Curvature of Complex Projective Manifolds]{On The Total Curvature and Betti Numbers of Complex Projective Manifolds}

\keywords{Total Curvature, Complex Projective Manifolds, Betti Number Estimates, Chern-Lashof Theorems}

\subjclass[2010] {Primary 53C55 Hermitian and K\"ahlerian Manifolds; Secondary 51N35 Questions of Classical Algebraic Geometry, 53C65 Integral Geometry}

\begin{abstract}
We prove an inequality between the sum of the Betti numbers of a complex projective manifold and its total curvature, and we characterize the complex projective manifolds whose total curvature is minimal.  These results extend the classical theorems of Chern and Lashof to complex projective space. 

\end{abstract}

\maketitle

%\tableofcontents

%%%%%%%%%%%%%%%%%%%%%%%%%%%%%%%%%%%%%%%%%%%%%%%%%%%%%%%%%%%%%%%%%%%%%%%%%%%%%%%%%%%%%%%%%%%%%%%%%%%%%%%%%%%%%%%%%%%%%%%%%

\section{Introduction}

\bigskip 

The central results of this paper are an inequality between the total curvature of a complex projective manifold and its Betti numbers and a characterization of the complex projective manifolds whose total curvature is minimal.  We will prove:   

\begin{theorem} 
\label{cpcl}

Let $M$ be a compact complex manifold, of complex dimension $m$, holomorphically immersed in complex projective space, and let $\mathcal{T}(M)$ be its total absolute curvature, as in Proposition \ref{cppc} and Definition \ref{cptc} below:  

\begin{itemize}

  \item[\textbf{A.}] Let $\beta_{i}$ be the Betti numbers of $M$ with real coefficients.  Then $\sum\limits_{i=0}^{2m}\beta_{i} \leq (\frac{m+1}{2})\mathcal{T}(M)$. 

\smallskip  

\begin{flushleft}
In particular, $\mathcal{T}(M) \geq 2$. 
\end{flushleft} 

  \bigskip

  \item[\textbf{B.}] If $\mathcal{T}(M) < 4$, then in fact $\mathcal{T}(M) = 2$.  This occurs precisely if $M$ is a linearly embedded complex projective subspace. 

\end{itemize}

\end{theorem}

\medskip  

The foundation for these results is a classical family of theorems that were proved by Chern and Lashof.  The definition of total absolute curvature for complex projective manifolds is based on an invariant which they defined, originally for submanifolds of Euclidean space, in \cite{CLI} and \cite{CLII}.  We will define the total absolute curvature of a complex projective manifold in Definition \ref{cptc} below.  However, we will prove that the total absolute curvature of a complex projective manifold has the following geometric meaning: 

\begin{proposition}
\label{cppc}

% Let $M$ be a compact complex manifold holomorphically immersed in complex projective space.  Let $\nu^{<\frac{\pi}{2}}M$ be its normal disk bundle of radius $\frac{\pi}{2}$, and let $Exp^{\perp}$ be the normal exponential map from $\nu^{<\frac{\pi}{2}}M$ to the ambient projective space $\C P^{N}$.  Let $\mathcal{T}(M)$ be its total absolute curvature.  Then:  
Let $M$ be a compact complex manifold holomorphically immersed in complex projective space.  Let $\nu^{<\frac{\pi}{2}}M$ be its normal disk bundle of radius $\frac{\pi}{2}$ and $Exp^{\perp}$ the normal exponential map from $\nu^{<\frac{\pi}{2}}M$ to the ambient projective space $\C P^{N}$.  Let $\mathcal{T}(M)$ be its total absolute curvature.  Then:  

\begin{center}
$\mathcal{T}(M) = \displaystyle \text{\footnotesize $\frac{2}{Vol(\C P^{N})}$}  \text{\LARGE $\int\limits_{\text{\scriptsize $\nu^{<\frac{\pi}{2}}M$}}$} \text{\LARGE $\vert$} det(dExp^{\perp}) \text{\LARGE $\vert$} dVol_{\text{\tiny $\nu^{<\frac{\pi}{2}}M$}} = \text{\footnotesize $\frac{2}{Vol(\C P^{N})}$} \displaystyle \text{\LARGE $\int\limits_{\text{\scriptsize $\C P^{N}$}}$} \sharp(Exp^{\perp})^{-1}(q) dVol_{\text{\tiny $\C P^{N}$}}.$
\end{center}

\end{proposition}  

\medskip 

In the first expression, we integrate over the normal disk bundle of radius $\frac{\pi}{2}$ because $\frac{\pi}{2}$ is the diameter of $\C P^{N}$.  In the latter, $\sharp(Exp^{\perp})^{-1}(q)$ denotes the pre-image count via $Exp^{\perp}$ for a point $q$ in $\C P^{N}$.  The proof of Proposotion \ref{cppc} shows that this result can also be taken as the definition of total absolute curvature for complex projective manifolds.  It is equivalent to the meaning of Chern and Lashof's invariant, which they defined for a submanifold $M$ of Euclidean space as follows: 

\begin{center}
\begin{equation}
\label{cltc}
\mathcal{T}(M) = \text{\footnotesize $\frac{1}{Vol(S^{N-1})}$} \displaystyle \text{\LARGE $\int\limits_{\text{\tiny $\nu^{1}M$}} \vert$} det(A_{\vec{u}}) \text{\LARGE $\vert$} dVol_{\text{\tiny $\nu^{1}M$}}.
\end{equation}
\end{center} 

\medskip

In this definition, $\nu^{1}M$ is the unit normal bundle of the immersion into $\R^{N}$ and $A_{\vec{u}}$ is the second fundamental form of the normal vector $\vec{u}$.  Dividing by $Vol(S^{N-1})$ ensures that $\mathcal{T}(M)$ is the same whether we regard $M$ as a submanifold of $\R^{N}$, or of a higher-dimensional space $\R^{N+N'}$ containing $\R^{N}$ as a linear subspace.  Chern and Lashof then proved:   

\begin{theorem}[Chern-Lashof Theorems, \cite{CLI, CLII}]
\label{cl}

Let $M$ be a closed manifold of dimension $n$ immersed in Euclidean space:   

\bigskip 

\begin{itemize}

  \item[\textbf{A.}] {\em (First Chern-Lashof Theorem)} Let $\beta_{i}$ the Betti numbers of $M$, with coefficients in the integers or any field.  Then $\sum\limits_{i=0}^{n}\beta_{i} \leq \mathcal{T}(M)$.  In particular, $\mathcal{T}(M) \geq 2$.  

  \bigskip

  \item[\textbf{B.}] {\em (Second Chern-Lashof Theorem)} If $\mathcal{T}(M) < 3$, then $M$ is homeomorphic to the $n$-sphere.  

  \bigskip

  \item[\textbf{C.}] {\em (Third Chern-Lashof Theorem)} $\mathcal{T}(M) = 2$ precisely if $M$ is a convex hypersurface in an $(n+1)$-dimensional linear subspace of $\R^{N}$.  

\end{itemize} 

\end{theorem}

\bigskip

In the introduction to their first paper on total curvature, \cite{CLI}, Chern and Lashof cite the theorems of Fenchel and F\'ary-Milnor, in \cite{FeI} and \cite{Fa, Mi}, as motivation for their results.  Fenchel's theorem states that a smooth closed curve $\gamma$ in $\R^{3}$ has total curvature at least $2\pi$, with equality precisely for plane convex curves.  The Chern-Lashof theorems can be understood as a far-reaching generalization of Fenchel's theorem, to compact Euclidean submanifolds of any dimension and codimension.  The F\'ary-Milnor theorem states that if $\gamma$ has total curvature at most $4\pi$, twice the minimum in Fenchel's theorem, then it is unknotted.  Part B of Theorem \ref{cpcl} gives a similar statement for complex projective manifolds. \\  

The Chern-Lashof theorems are also related to an extrinsic formulation of the Gauss-Bonnet-Chern theorem, for submanifolds of Euclidean space, which was proven by Fenchel in \cite{FeII} and Allendoerfer in \cite{Al}.  We will show in Theorem \ref{cpgbc} that Allendoerfer and Fenchel's theorem about submanifolds of Euclidean space has a parallel for complex projective manifolds. \\ 

The definition of total absolute curvature for complex projective manifolds is an adaptation of Chern and Lashof's invariant - like their invariant, it depends only on the second fundamental form of a complex submanifold of projective space: \\ 

\begin{definition}[Total Absolute Curvature of Complex Projective Manifolds]
\label{cptc}

{\em Let $M$ be a complex manifold, of complex dimension $m$, holomorphically immersed in the complex projective space $\C P^{N}$.  Its total absolute curvature $\mathcal{T}(M)$ is defined to be: }

\smallskip 

\begin{center}
$\mathcal{T}(M) = \displaystyle \text{\footnotesize $\frac{2}{Vol(\C P^{N})}$} \text{\LARGE $\int\limits_{\text{\scriptsize $\nu^{<\frac{\pi}{2}}M$}}$}  \text{\LARGE $\vert$} det \text{\LARGE $($} \cos(r) Id_{T_{p}M} - \left(\frac{\sin r}{r}\right) A_{\vec{v}}  \text{\LARGE $) \vert$} \cos(r) \left(\frac{\sin r}{r}\right)^{\tiny (2N-2m-1)} dVol_{ \text{\tiny $\nu^{<\frac{\pi}{2}}M$}}(\vec{v}),  $ 
\end{center}

\smallskip

{\em where $\nu^{ {\tiny < \frac{\pi}{2}}}M$ is the normal disk bundle of radius $\frac{\pi}{2}$, $A_{\vec{v}}$ is the second fundamental form of the normal vector $\vec{v}$, $r$ is its norm, and $Id_{T_{p}M}: T_{p}M \rightarrow T_{p}M$ is the identity transformation of the tangent space to $M$ at its base point $p$.}  

\end{definition} 

\bigskip  

Dividing by the volume of $\C P^{N}$ in Definition \ref{cptc} ensures that $\mathcal{T}(M)$ is the same whether we regard $M$ as immersed in $\C P^{N}$, or in a higher-dimensional space $\C P^{N+N'}$ containing $\C P^{N}$ as a linear subvariety.  The extra factor of $2$ in Definition \ref{cptc} will be explained in Proposition \ref{tcsame} and Remark \ref{mindifference}.  When we need to distinguish the invariant defined for smooth submanifolds of Euclidean space in (\ref{cltc}) from the invariant for complex projective manifolds in Definition \ref{cptc}, we will write the first as $\mathcal{T}_{\R^{N}}(M)$ and the second as $\mathcal{T}_{\C P^{N}}(M)$.  In Definition \ref{atcsphere}, we will give a formula for the total absolute curvature of submanifolds of spheres, which we denote $\mathcal{T}_{S^{N}}(M)$.  \\ 

Chern and Lashof's invariant depends on the extrinsic geometry of a submanifold of Euclidean space.  However, Calabi proved in \cite{Ca} that if a K\"ahler manifold (with a fixed metric) admits a holomorphic isometric immersion into complex projective space, even locally, then any two holomorphic isometric immersions of this manifold into complex projective space are congruent by a holomorphic isometry of the ambient space.  It follows that the total absolute curvature of a complex projective manifold is actually part of its intrinsic geometry.  It would thus be interesting to find a completely intrinsic representation of the total absolute curvature of a complex projective manifold.  We will do so for curves in the complex projective plane: 

\begin{theorem} Let $\Sigma$ be a smooth curve in $\C P^{2}$, with $K$ the sectional curvature of its projectively induced metric.  Then:  
\label{curveformula} 

\bigskip  
\begin{center}
$\mathcal{T}(\Sigma) = \displaystyle \frac{1}{\pi} \text{\LARGE $\int\limits_{\text{\small $\Sigma$}}$}  \frac{(K-4)^{2} + 4}{6 - K} dA_{\Sigma}. $
\end{center}

\end{theorem}  

\bigskip  

Note that if $\Sigma$ is as above, then its sectional curvature is bounded above by $4$, the holomorphic sectional curvature of the ambient space.  This implies that the integrand in Theorem \ref{curveformula} is well-defined.  Theorem \ref{curveformula} is a corollary of a more general result about the total absolute curvature of complex projective hypersurfaces which we will state and prove in Theorem \ref{hypatc}.  Theorem \ref{curveformula} implies several results about the total absolute curvature of smooth plane curves, which we state and prove in Proposition \ref{curveestimate}.  In particular, the total absolute curvature of such a curve determines its degree: 

\begin{proposition} 
\label{atcdeterminesdegree}

Let $\Sigma$ be a smooth curve in $\C P^{2}$.  Then the degree of $\Sigma$ is the unique natural number $d$ such that $2d^{2} - 4d + 4 \leq \mathcal{T}(\Sigma) \leq 2d^{2}$.

\end{proposition}

\bigskip 

In Example \ref{fermatconic}, we will also use Theorem \ref{curveformula} to show that Part B of Theorem \ref{cpcl} is sharp. \\  

We can summarize the proof of Theorem \ref{cpcl} as follows: \\ 

In Proposition \ref{tcsame}, we will prove that if $M$ is a complex manifold holomorphically immersed in the complex projective space $\C P^{N}$, and $\widetilde{M}$ is the $S^{1}$ bundle over $M$ which is induced by the Hopf fibration $\pi: S^{2N+1} \rightarrow \C P^{N}$, and we immerse $\widetilde{M}$ in $S^{2N+1}$ by lifting the immersion of $M$ into $\C P^{N}$, then $\mathcal{T}_{S^{2N+1}}(\widetilde{M}) = \mathcal{T}_{\C P^{N}}(M)$. \\ 

Chern and Lashof's proof of Theorem \ref{cl} can be generalized to give similar theorems about submanifolds of any symmetric space - this was discovered by Koike in \cite{KoI} and \cite{KoII}.  However, in almost all cases, these results are much narrower than those in Chern and Lashof's original results.  In particular, in complex projective space, these theorems only apply to submanifolds of real dimension $1$ or less.  In spheres, on the other hand, Chern and Lashof's proofs can be adapted to give results that are equivalent to their theorems in full generality.  We will state and prove these results in Section \ref{sphereresults}.  Because Chern and Lashof's proofs work in such generality for submanifolds of spheres, and break down so completely in complex projective space, our strategy for proving the main theorems in this paper is to relate the geometry and topology of a complex projective manifold to those of its pre-image via the Hopf fibration in the sphere. \\  

We now give an outline of this paper: \\  

In Section \ref{sphereresults}, we will review previous research related to total curvature, especially in the complex projective setting.  In addition to the Chern-Lashof theorems, this includes the results of Weyl on tube volumes in \cite{We}, of Allendoerfer and Fenchel on the higher-dimensional Gauss-Bonnet theorem in \cite{Al} and \cite{FeII}, and the more recent work of several authors.  We will also state and prove a formulation of the Chern-Lashof theorems for submanifolds of spheres. \\ 

In Sections \ref{bettis} and \ref{minimizers}, we prove the main results in this paper:  In Section \ref{bettis}, we will prove Part A of Theorem \ref{cpcl}.  This result follows from several other inequalities between the total curvature and Betti numbers of complex projective manifolds, which are generally stronger than Theorem \ref{cpcl}.A - we will state and prove these results in Propositions \ref{basicestimate} and \ref{detailedestimate}.  In Section \ref{minimizers}, we prove Part B of Theorem \ref{cpcl} and discuss its relationship to several other results in the geometry of complex projective manifolds. \\  

In Section \ref{tfgbc}, we will prove Proposition \ref{cppc} and Theorem \ref{cpgbc}.  We will discuss the relationship between the results of this paper and some foundational results in the theory of total curvature in greater detail, building on the discussion in Section \ref{sphereresults}. \\  

In Section \ref{hypersurfaces}, we will study the total absolute curvature of complex projective hypersurfaces (complex projective manifolds of complex codimension $1$) in greater detail.  We will prove Theorem \ref{curveformula} and discuss its implications. \\  

Throughout the paper, we will discuss possible directions for future research based on these results. \\ 

Throughout this paper, unless stated otherwise, $\R^{N}$ and $\C P^{N}$ will carry their canonical metrics, with the metric on $\C P^{N}$ normalized to have holomorphic sectional curvature $4$.  Spheres will likewise carry their canonical metrics with constant curvature $1$.  Standard results and formulas from complex and algebraic geometry used in this paper can be found in \cite{Hb} and several other texts.  Background about the local differential geometry of complex and K\"ahler submanifolds can be found in \cite{Gr}. \\ 

\textbf{Acknowledgments:} I am deeply grateful to Christopher Croke, my advisor, for his mentorship, support and encouragement, and for introducing me to the mathematics that led to these results.  I am very happy to thank Herman Gluck, Peter McGrath, Tony Pantev, Brian Weber and Wolfgang Ziller, for many helpful and enjoyable conversations.  I would like to thank the referee for their careful reading of the manuscript, their helpful suggestions for improving the exposition and their insightful questions about the relationship between these results and earlier work.  

%%%%%%%%%%%%%%%%%%%%%%%%%%%%%%%%%%%%%%%%%%%%%%%%%%%%%%%%%%%%%%%%%%%%%%%%%%%%%%%%%%%%%%%%%%%%%%%%%%%%%%%%%%%%%%%%%%%%%%%%%

\section{The Chern-Lashof Theorems and Their Spherical Formulations}
\label{sphereresults}

The proof of the Chern-Lashof theorems in \cite{CLI} and \cite{CLII} combines Morse theory, integral geometry, and a careful analysis of a Euclidean submanifold's local extrinsic geometry.  For any manifold $M$ immersed in Euclidean space, of any dimension and codimension, one can define the Gauss map on the unit normal bundle $\nu^{1}M$ of the immersion.  This map sends normal vectors $\vec{u}$ to their parallels in the unit sphere.  A careful analysis of this map implies that for almost all $\vec{w}$ in the unit sphere, the height function $h_{\vec{w}}$ in the direction $\vec{w}$, when restricted to $M$, is a Morse function. \\

The upper bound for the sum of the Betti numbers in the first Chern-Lashof theorem is based on the fact that if $f$ is a Morse function on a compact manifold $M$, counting its critical points gives an upper bound for the sum of the Betti numbers of $M$.  The total absolute curvature of $M$ is equal to the average number of critical points of the height functions $h_{\vec{w}}$, which are Morse functions for almost all $\vec{w}$ in the unit sphere, as described above.  This average is greater than or equal to the sum of the Betti numbers of $M$.  The second Chern-Lashof theorem is based on a theorem of Reeb, in \cite{Re}: that a compact manifold which supports a Morse function with only two critical points is homeomorphic to a sphere. \\ 

These observations are related to an extrinsic formulation of the Gauss-Bonnet-Chern theorem, which was proved by Fenchel and Allendoerfer in \cite{FeII} and \cite{Al}, building on earlier work of H. Hopf in ~\cite{Ho}.  This formulation of the Gauss-Bonnet-Chern theorem can also be proved from the same observations as the Chern-Lashof theorems.  In this proof, the Euler characteristic of a submanifold $M$ of Euclidean space arises as the sum of the critical points of the Morse functions $h_{\vec{w}}$, with each critical point signed by its index: \\

\begin{center}

$\displaystyle \chi(M) = \sum\limits_{i = 0}^{n}(-1)^{i}\beta_{i}(M) = \sum\limits_{i=0}^{n} (-1)^{i}C_{i},$ 

\end{center} 

\medskip

where $\beta_{i}(M)$ is the $i^{th}$ Betti number of $M$, with coefficients in the integers or any field, and $C_{i}$ is the number of critical points of $h_{\vec{w}}$ of index $i$.  The corresponding fact in the Chern-Lashof theorems is that the sum of the Betti numbers of $M$ can be bounded above by counting the critical points of a Morse function on $M$ without regard to their index: \\

\begin{center}

$\displaystyle \sum\limits_{i = 0}^{n}\beta_{i}(M) \leq \sum\limits_{i = 0}^{n} C_{i}.$ 

\end{center} 

\medskip

The extrinsic formulation of the Gauss-Bonnet-Chern theorem described above then says: 

\begin{equation}
\label{egbc}
\displaystyle \chi(M) = \sum\limits_{i = 0}^{n}(-1)^{i}\beta_{i}(M) = \text{\scriptsize $\frac{1}{Vol(S^{N-1})}$} \text{\LARGE $\int\limits_{\text{\tiny $\nu^{1}M$}}$} det(A_{\vec{u}}) dVol_{\text{\tiny $\nu^{1}M$}}(\vec{u}).  
\end{equation}

\smallskip

The first Chern-Lashof theorem says: \\ 

\begin{center}

$\displaystyle \sum\limits_{i = 0}^{n}\beta_{i}(M) \leq \text{\scriptsize $\frac{1}{Vol(S^{N-1})}$} \text{\LARGE $\int\limits_{\text{\tiny $\nu^{1}M$}} \vert$} det(A_{\vec{u}}) \text{\LARGE $\vert$} dVol_{\text{\tiny $\nu^{1}M$}}(\vec{u}). $

\end{center} 

\smallskip

The result above is genuinely the Gauss-Bonnet-Chern theorem in that for even-dimensional manifolds $M$, the integrand in (\ref{egbc}) coincides with the Pfaffian of the curvature forms, up to a normalization for the dimension and codimension of $M$ (for odd-dimensional manifolds, it is zero.)  The proof of the Gauss-Bonnet-Chern theorem described above can be found in \cite{Ba}.  The book \cite{Wi} by Willmore gives an overview of several branches of geometry in which the Chern-Lashof theorems have had a strong influence. \\  

To our knowledge, the first results about total absolute curvature in the complex projective setting are those of Ishihara in \cite{I}.  However, Milnor's work in \cite{MiII} and Thom's work in \cite{Th} both use Morse theory to establish an inequality between the degrees of real algebraic varieties and their Betti numbers.  Based on their comments, both authors seem to have considered the possibility of extending their results using some of the observations we have used here, and Milnor uses some of these observations to extend his results to complex projective varieties.  Cecil's results in \cite{Ce} also involve the application of Morse theory to complex projective manifolds and use some of the same observations as the Chern-Lashof theorems. \\  

Ishihara's work gives a definition of total absolute curvature for a submanifold $M$ of complex projective space, and then relates this invariant to a family of maps from $M$ to complex projective lines in the ambient space.  In \cite{AN}, Arnau and Naveira define a family of total curvature invariants for submanifolds of complex projective space, which are adaptations of invariants defined for submanifolds of Euclidean space by Santal\'o in \cite{SaI, SaII}.  One of the invariants defined by Santal\'o coincides with the invariant defined by Chern and Lashof, and Arnau and Naveira show that Ishihara's invariant coincides with one of their invariants. \\ 

In \cite{KoI} and \cite{KoII}, Koike pursued adaptations of the Chern-Lashof theorems, and of the extrinsic formulation of the Gauss-Bonnet-Chern theorem, to submanifolds isometrically immersed in any symmetric space.  He observed that the essential meaning of Chern and Lashof's invariant in (\ref{cltc}) for submanifolds of Euclidean space is reflected in a similarly-defined invariant for submanifolds of all symmetric spaces.  In particular, for submanifolds of complex projective spaces, Koike's invariant in (1.8) of \cite{KoII} coincides with Definition \ref{cptc}, and for submanifolds of spheres, this invariant coincides with Definition \ref{atcsphere} below.  Koike's work also shows that the relationship between total curvature and the normal exponential map described in Proposition \ref{cppc} holds in any symmetric space. \\ 

In addition to the Gauss-Bonnet-Chern theorem and these results about total absolute curvature, there is another family of results that we believe provide important background for the results in this paper: in \cite{We}, Hermann Weyl proved that if $M$ is a compact Riemannian manifold isometrically embedded in Euclidean space, the volume of a small tubular neighborhood of $M$ depends only on its intrinsic geometry, not on the embedding.  More precisely, he proved: 

\begin{theorem}[Weyl's Tube Formula, \cite{We}] 
\label{wtf}

Let $M$ be an $n$-dimensional compact Riemannian manifold isometrically immersed in $\R^{N}$, and let $\nu^{<r}M$ be its normal disk bundle of radius $r$.  Let $Exp^{\perp}$ be the normal exponential map from $\nu^{<r}M$ to $\R^{N}$.  Then: 

\begin{equation}
\label{wformula}
\displaystyle \text{\LARGE $\int\limits_{\nu^{<r}M}$} ( Exp^{\perp})^{*}(dVol_{\R^{N}}) = \frac{(\pi r^{2})^{\frac{N-n}{2}}}{\Gamma(\frac{N-n+2}{2})} \sum\limits_{i=0}^{\lfloor \frac{n}{2} \rfloor} \frac{K_{2i}(M)r^{2i}}{(N-n+2)(N-n+4) \cdots (N-n+2i)}.
\end{equation}
\medskip

In the formula above, the coefficients $K_{2i}(M)$ are invariants of the curvature tensor of $M$. \\  

If $M$ is embedded in $\R^{N}$ and $r$ is chosen small enough that $Exp^{\perp}: \nu^{<r}M \rightarrow \R^{N}$ is injective, then the formula above gives the volume of the tube of radius $r$ about $M$.  This implies that, in this situation, the tube volume depends only on the intrinsic geometry of $M$, not on the embedding.  

\end{theorem}

Weyl's proof of Theorem \ref{wtf} generalizes immediately to submanifolds of spheres, and in \cite{We} he states and proves his result for submanifolds of spheres as well as Euclidean space.  There is also a tube formula for complex submanifolds of complex projective space.  To our knowledge, this was first found independently by Wolf in \cite{Wo} and Flaherty in \cite{Fl}.  Weyl's tube formula was extended to compatible submanifolds of all rank-$1$ symmetric spaces, including complex submanifolds of $\C P^{N}$, by Gray and Vanhecke in \cite{GV}.  In \cite{GrII}, Gray showed how to express the tube volume about a complex projective manifold in terms of its Chern forms - this result will be important in the proof of Theorem \ref{cpgbc}. \\ %The tube formula for complex projective manifolds can also be found in Alfred Gray's book \cite{Gr}. \\  

Weyl's tube formula, the Gauss-Bonnet-Chern theorem and the Chern-Lashof theorems are closely related: for complex projective manifolds and submanifolds of spheres, the total absolute curvature integrand coincides with the integrand in the tube formula for sufficiently small normal vectors.  In general, the total absolute curvature integrand is the absolute value of the integrand in the tube formula - we explain this in more detail in Section \ref{tfgbc}.  The highest-degree term in Weyl's tube formula is, up to a normalization, the Gauss-Bonnet-Chern integral, and the top-degree coefficient in (\ref{wformula}) is therefore given by the Euler characteristic of $M$ - this fact played an essential part in Allendoerfer's proof the Gauss-Bonnet-Chern theorem in \cite{Al}.  The history of this development can be found in \cite{Gr}. \\ 

It is the relationship between the Chern-Lashof theorems, the Gauss-Bonnet-Chern theorem and Weyl's tube formula, and the existence of a parallel to Weyl's tube formula for complex projective manifolds in the results mentioned above, that motivated our search for an extension of the Chern-Lashof theorems to complex projective space.  We believe it would also be interesting to derive Gray and Vanhecke's tube formulas for compatible submanifolds of complex and quaternionic projective space in \cite{GV} from Weyl's tube formula for submanifolds of spheres in \cite{We} using Hopf fibrations, as we have done here. \\ 

Because our results for complex projective manifolds depend on a formulation of the Chern-Lashof theorems for submanifolds of spheres, we will state and prove these results in the remainder of this section.  The majority of these results are not new, but they are used in the proofs of our main theorems, and their proofs are helpful in many of our proofs in the rest of the paper.  We will therefore give complete proofs of these results.  The reader can also go ahead to the beginning of Section \ref{bettis} and return to these results as needed. \\  

The total absolute curvature of a submanifold of a sphere is defined as follows: 
 
\begin{definition}[Total Absolute Curvature of Submanifolds of Spheres]  
\label{atcsphere}
{\em Let $M$ be an $n$-dimensional differentiable manifold immersed in the sphere $S^{N}$.  The total absolute curvature of $\mathcal{T}(M)$ of $M$ is:} \\ 

\begin{center}
$\mathcal{T}(M) = \frac{1}{Vol(S^{N})} \displaystyle \text{\LARGE $\int\limits_{\text{\scriptsize $\nu^{<\pi}M$}}$}  \text{\LARGE $\vert$} det \text{\LARGE $($} \cos(r) Id_{T_{p}M} - \left( \frac{\sin r}{r} \right) A_{\vec{v}} \text{\LARGE $) \vert$} \left( \frac{\sin r}{r}\right)^{(N-n-1)}  dVol_{\text{\tiny $\nu^{<\pi}M$} }(\vec{v}).  $
\end{center}

\end{definition}  

We integrate over $\nu^{<\pi}M$, the bundle of normal vectors of length less than $\pi$, because $\pi$ is the diameter of $S^{N}$.  $A_{\vec{v}}$ denotes the second fundamental form of the normal vector $\vec{v}$, $r$ denotes $||\vec{v}||$, and $Id_{T_{p}M}: T_{p}M \rightarrow T_{p}M$ is the identity transformation of the tangent space to $M$ at its basepoint $p$.  As in Definition \ref{cptc}, the normalization by $Vol(S^{N})$ in Definition \ref{atcsphere} ensures that $\mathcal{T}_{S^{N}}(M) = \mathcal{T}_{S^{N+N'}}(M)$ if $S^{N}$ is embedded as a totally geodesic submanifold of $S^{N+N'}$. \\   

For submanifolds of spheres, the differential of the normal exponential map  can be expressed entirely in terms of the second fundamental form: let $\vec{u}$ be a unit normal vector to a manifold $M$ immersed in $S^{N}$ and let $A_{\vec{u}}$ be its second fundamental form.  Let $e_{1}, ... , e_{n}$ be a set of principal vectors for $A_{\vec{u}}$ with principal curvatures $\kappa_{1}, ... , \kappa_{n}$, and let $u_{2}, ... , u_{N-n}$ be an orthonormal basis for the subspace of $\nu_{p}M$ orthogonal to $\vec{u}$.  Then the differential of the normal exponential map at a normal vector $\vec{v} = r\vec{u}$ can be represented in this basis using the Jacobi fields of the sphere, as follows: \\

\begin{itemize}

\item $dExp^{\perp}_{\vec{v}}(e_{i}) = \left( \cos(r) - \kappa_{i} \sin(r) \right) E_{i}(r)$, where $E_{i}$ is the parallel vector field with initial value $e_{i}$ along the geodesic $\gamma_{\vec{u}}$ of $S^{N}$ through $\vec{u}$, corresponding to the Jacobi field $J_{i}(t)$ along $\gamma_{\vec{u}}$ with $J_{i}(0) = e_{i}$ and $J_{i}'(0) = -\kappa_{i} e_{i}$.  

\bigskip

\item $dExp^{\perp}_{\vec{v}}(u_{j}) = \left(\frac{\sin r}{r}\right) F_{j}$, where $F_{j}$ is the parallel vector field along $\gamma_{\vec{u}}$ with initial value $u_{j}$, corresponding to the Jacobi field $J_{j}(t)$ along $\gamma_{\vec{u}}$ with $J_{j}(0) = 0$ and $J_{j}'(0) = \frac{1}{r}u_{j}$.

\bigskip

\item $dExp^{\perp}_{\vec{v}}(\vec{u}) = \gamma_{\vec{u}}'(r)$, corresponding to the Jacobi field $\gamma_{\vec{u}}'(r)$.  

\end{itemize}

\bigskip

This gives us: 

\begin{center}
$det(dExp^{\perp})_{\vec{v}} =  \prod\limits_{i=1}^{n} \left( \cos(r) - \kappa_{i} \sin(r) \right) \left(\frac{\sin r}{r}\right)^{(N-n-1)}$

\begin{equation}
\label{detdexpperp}
= det \left( \cos(r) Id_{T_{p}M} - \sin(r) A_{\vec{u}} \right) \left( \frac{\sin r}{r}\right)^{(N-n-1)}.
\end{equation}
\end{center} 

\medskip

Integrating $\vert det(dExp^{\perp}) \vert$ over a measurable subset of $\nu^{<\pi}M$ defines a positive measure on $\nu^{<\pi}M$ which is, in a natural sense, the pull-back of the measure on the ambient space $S^{N}$ via $Exp^{\perp}$.  Equation (\ref{detdexpperp}) implies that, up to the normalization by $Vol(S^{N})$, $\mathcal{T}(M)$ is the total mass of $\nu^{<\pi}M$ with this measure.  We record this result in the following proposition: 

\begin{proposition}
\label{spheretcmeaning}  Let $M$ be a closed manifold immersed in the sphere $S^{N}$.  Then: \\ 

$\mathcal{T}(M) =  \frac{1}{Vol(S^{N})} \displaystyle \text{\LARGE $\int\limits_{\text{\footnotesize $\nu^{<\pi}M$}} \vert$} det(dExp^{\perp}_{\vec{v}}) \text{\LARGE $\vert$} dVol_{\text{\tiny $\nu^{<\pi}M$}}(\vec{v}). $

\end{proposition}

\medskip 

In Chern and Lashof's original theorems, the equivalent statement is that for a submanifold $M$ of $\R^{N}$, up to normalization by $Vol(S^{N-1})$, $\mathcal{T}(M)$ is the total mass of the unit normal bundle $\nu^{1}M$, with the positive measure pulled back from the unit sphere $S^{N-1}$ by the Gauss map. \\  

It will be helpful to note that the total absolute curvature of a spherical submanifold can also be written as follows: 

\begin{equation}
\label{atcsphereformula}
\mathcal{T}(M) = \text{\scriptsize $\frac{1}{Vol(S^{N})}$} \displaystyle \text{\LARGE $\int\limits_{\text{\footnotesize $\nu^{1}M$}} \int\limits_{\text{\footnotesize $0$}}^{\text{\footnotesize $\pi$}} \vert$} \sum_{i=0}^{n} (-1)^{i}\sin^{(N-n-1+i)}(r) \cos^{(n-i)}(r) \sigma_{i}( \small{\kappa} ) \text{\LARGE $\vert$} \text{\em dr $dVol_{\text{\tiny $\nu^{1}M$}}(\vec{u})$}.  
\end{equation}
\medskip

Here, $\sigma_{i}( \kappa )$ represents the $i^{th}$ elementary symmetric function of the principal curvatures of the normal vector $\vec{u}$. In particular, $\sigma_{1}( \kappa )$ is the mean curvature $\kappa_{1} + \kappa_{2} + ... + \kappa_{n}$, $\sigma_{2}( \kappa ) = \kappa_{1}\kappa_{2} + \kappa_{1}\kappa_{3} + ... + \kappa_{n-1}\kappa_{n}$ and $\sigma_{n}( \kappa ) = \kappa_{1}\kappa_{2}\dots\kappa_{n}$ is the Gauss curvature in the direction $\vec{u}$, etc. \\   

The spherical formulations of the Chern-Lashof theorems are as follows.  We include the proofs below for completeness.  These results can also be found in the work of Koike: 

\begin{theorem}[Spherical Formulation of the Chern-Lashof Theorems - see \cite{KoII}]
\label{spherecl}

Let $M$ be an $n$-dimensional closed manifold isometrically immersed in the sphere $S^{N}$: 

\begin{itemize}

  \item[\textbf{A.}]  Let $\beta_{i}$ be the $i^{th}$ Betti number of $M$ with coefficients in the integers or any field.  Then $\sum\limits_{i=0}^{n} \beta_{i} \leq \mathcal{T}(M)$.  In particular, $\mathcal{T}(M) \geq 2$.  

  \bigskip  

  \item[\textbf{B.}]  If $\mathcal{T}(M) < 3$, then $M$ is homeomorphic to $S^{n}$.  

  \bigskip  

  \item[\textbf{C.}] $\mathcal{T}(M) = 2$ precisely if $M$ is the boundary of a geodesic ball in an $(n+1)$-dimensional totally geodesic subsphere of $S^{N}$.  

\end{itemize}

\end{theorem} 

\medskip

The spherical Chern-Lashof theorems are based on the observation that if $M$ is a closed manifold isometrically immersed in a round sphere $S^{N}$, then for almost all $q \in S^{N}$, the distance function from $q$ is a Morse function when restricted to $M$.  The equivalent fact in the classical Chern-Lashof theorems is that for almost all unit vectors $\vec{w}$ in $\R^{N}$, the height function $h_{\vec{w}}$ is Morse on $M$.  As with submanifolds of Euclidean spaces, the total absolute curvature of a spherical submanifold is the average number of critical points of a family of Morse functions on $M$ - in this case, of the distance functions $d_{q}$.  This average gives an upper bound for the sum of the Betti numbers of $M$, and if it is less than $3$, it implies $M$ is homeomorphic to a sphere by Reeb's theorem.  \\   

We establish that almost all distance functions are Morse functions on $M$, and we derive a formula for their Hessians, in the next result:  

\begin{proposition}  
\label{spherefullmeasure}
Let $M$ be a closed manifold immersed in the sphere $S^{N}$ as in Theorem \ref{spherecl}.  For almost all $q \in S^{N}$, the distance function from $q$, when restricted to $M$, is a smooth Morse function.  \\

\end{proposition}

\begin{proof}[Proof of Proposition ~\ref{spherefullmeasure}]   

For the proof of this proposition and Lemma \ref{hesslemma} below, we let $\widetilde{d_{q}}$ denote $dist(q, \cdot)$ as a function on $S^{N}$, and we let $d_{q}$ denote its restriction to $M$.  For $q$ not in the cut locus of any $p \in M$ (i.e. for $q \not\in \pm M$), $d_{q}$ is a smooth function on $M$.  That it is almost always a Morse function comes from the following:  

\begin{lemma}
\label{hesslemma}
A point $p$ of $M$ is critical for $d_{q}$ iff $q = Exp^{\perp}(\vec{v})$ for some $\vec{v} \in \nu_{p}^{<\pi}M$.  \\

In that case, letting $r = ||\vec{v}||$ and $\vec{u} = \frac{\vec{v}}{r}$, $Hess(d_{q})$ is diagonolized at $p$ by a set of principal directions for the second fundamental form $A_{\vec{u}}$, and the eigenvalue of $Hess(d_{q})$ corresponding to the principal curvature $\kappa$ is $\cot(r) - \kappa$.  

\end{lemma}

\begin{proof}[Proof of Lemma ~\ref{hesslemma}]

The gradient of $d_{q}$ on $M$ is the orthogonal projection of the gradient of $\widetilde{d_{q}}$ on $S^{N}$.  The gradient of $\widetilde{d_{q}}$ at a point $\bar{q}$ in $S^{N} \setminus \lbrace q, -q \rbrace$ is tangent to the minimizing geodesics from $q$ to $\bar{q}$.  $grad(d_{q})$ is zero precisely where the minimizing geodesic from $q$ is normal to $M$, so that $q$ is in the image of the normal exponential map from $p$: \\

\begin{center}
$q = Exp^{\perp} \text{\Large $($} -\widetilde{d_{q}}(p)grad(\tilde{d_{q}}) \text{\Large $)$}.$
\end{center}

\bigskip

Similarly, if $Exp^{\perp}(\vec{v}) = q$, then the minimizing geodesic from $q$ to $p$ is normal to $M$ at $p$, and $p$ is critical for $d_{q}$.  \\

If $p \in M$ and $\vec{u} \in \nu_{p}^{1}M$ are as above, let $e_{1}, ... , e_{n}$ be a set of principal vectors for $A_{\vec{u}}$.  For $0 < r < \pi$, let $q = Exp^{\perp}(r\vec{u})$.  We let $grad(\widetilde{d_{q}})$ and $grad(d_{q})$ denote the vector fields on $S^{N}$ and $M$ respectively, as above.  Because $grad(d_{q}) = grad(\widetilde{d_{q}})^{\top}$, along $M$ we have: \\

\begin{center}
$grad(\widetilde{d_{q}}) = grad(d_{q}) + grad(\widetilde{d_{q}})^{\perp}.$   
\end{center}

\bigskip  

A geodesic sphere of radius $r$ in $S^{N}$ has principal curvature $\cot(r)$ in any tangent direction, relative to the outward unit normal, so $\nabla^{S^{N}}_{e}grad(\widetilde{d_{q}}) = \cot (r) e$ for any $e \in T_{p}M$, which will also be tangent to the geodesic sphere of radius $r$ about $q = Exp^{\perp}(r\vec{u})$. \\  

Letting $F$ be $-\frac{grad(\widetilde{d_{q}})^{\perp}}{||grad(\widetilde{d_{q}})^{\perp}||}$ , $F$ is a unit normal vector field to $M$ in a neighborhood of $p$, which coincides with $\vec{u}$ at $p$, so for each principal direction $e_{i}$ as above, \\ 

\begin{center}
$(\nabla_{e_{i}}^{S^{N}}F)^{\top} = -\kappa_{i}e_{i}.$
\end{center}

\bigskip  

We also have the following: \\ 

\begin{center}
$\nabla_{e_{i}}^{S^{N}}F = \text{\LARGE $($} \frac{e_{i}(||grad(\widetilde{d_{q}})^{\perp}||)}{||grad(\widetilde{d_{q}})^{\perp}||^{2}} \text{\LARGE $)$} grad(\widetilde{d_{q}})^{\perp} - \text{\Large $($} \frac{1}{||grad(\widetilde{d_{q}})^{\perp}||} \text{\Large $)$} \nabla_{e_{i}}^{S^{N}}grad(\widetilde{d_{q}})^{\perp}.$ 
\end{center}

\bigskip  

The tangential part of $\nabla_{e_{i}}^{S^{N}}F$ is therefore the tangential part of the expression above: \\ 

\begin{center}

$(\nabla_{e_{i}}^{S^{N}}F)^{\top} = \text{\LARGE $($} -\frac{1}{||grad(\widetilde{d_{q}})^{\perp}||}\nabla_{e_{i}}^{S^{N}}(grad(\widetilde{d_{q}})^{\perp})\text{\LARGE $)$}^{\top} = \text{\LARGE $($} -\frac{1}{||grad(\widetilde{d_{q}})^{\perp}||}\nabla_{e_{i}}^{S^{N}} \text{\Large $($}  grad(\widetilde{d_{q}}) - grad(d_{q}) \text{\Large $)$}  \text{\LARGE $)$}^{\top}$ 

\bigskip  

$ = (\frac{1}{||grad(\widetilde{d_{q}})^{\perp}||}) \left( \nabla_{e_{i}}^{M}grad(d_{q}) - \cot (r) e_{i} \right) = (\frac{1}{||grad(\widetilde{d_{q}})^{\perp}||}) \left( Hess(d_{q})(e_{i}) - \cot (r) e_{i} \right).  $ 

\end{center}
\bigskip  

Noting that $||grad(\widetilde{d_{q}})^{\perp}|| = 1$ at critical points of $d_{q}$, we then have $-\kappa_{i}e_{i} = Hess(d_{q})(e_{i}) - \cot(r)e_{i}$, and therefore that $Hess(d_{q})(e_{i}) = (\cot r - \kappa_{i})e_{i}$.  This completes the proof of Lemma \ref{hesslemma}.  

\end{proof}

The lemma implies that $q = Exp^{\perp}(r\vec{u})$ has a degenerate critical point at $p$ if and only if $r = \arccot \kappa$ for $\kappa$ a principal curvature of $M$ in the direction $\vec{u}$.  This is the same as the condition that the normal exponential map have a critical point at $r\vec{u}$, by the expression for $det(dExp^{\perp})$ in (\ref{detdexpperp}).  The $q \in S^{N}$ for which $d_{q}$ is not Morse are therefore the focal points of $M$ in $S^{N}$.  These are the critical values of the normal exponential map, which are of measure zero in $S^{N}$ by Sard's Theorem.  \\ 

This completes the proof of Proposition ~\ref{spherefullmeasure}.  

\end{proof}

\begin{proof}[Proof of Parts A and B of Theorem \ref{spherecl}]

By Proposition ~\ref{spherefullmeasure} and the lemma in its proof, the regular values of $Exp^{\perp} : \nu^{<\pi}M \rightarrow S^{N}$ are precisely the points $q \in S^{N}$ for which $d_{q}$ is a Morse function.  Letting $S^{N}_{reg}$ denote this subset of $S^{N}$, Sard's Theorem implies that $S^{N}_{reg}$ is of full measure in $S^{N}$.  In fact, $S^{N}_{reg}$ also contains an open, dense subset of $S^{N}$. \\  

This can be seen by extending $Exp^{\perp}$ to be defined on the bundle $\widehat{\nu}M$ over $M$, whose fibre at $p$ is the totally geodesic $(N-n)$-dimensional subsphere of $S^{N}$ orthogonal to $M$ at $p$.  We denote this extension $\widehat{Exp}^{\perp}: \widehat{\nu}M \rightarrow S^{N}$.  The critical points of  $\widehat{Exp}^{\perp}: \widehat{\nu}M \rightarrow S^{N}$ are closed in the compact manifold $\widehat{\nu}M$, and thus a compact subset of $\widehat{\nu}M$.  Their image, the critical values of $\widehat{Exp}^{\perp}: \widehat{\nu}M \rightarrow S^{N}$, are a compact, and thus a closed subset of $S^{N}$, whose complement is of full measure by Sard's Theorem.  The critical values of $Exp^{\perp}: \nu^{<\pi}M \rightarrow S^{N}$ are a subset of those of $\widehat{Exp}^{\perp}: \widehat{\nu}M \rightarrow S^{N}$.  The regular values of $Exp^{\perp}: \nu^{<\pi}M \rightarrow S^{N}$ therefore contain the regular values of $\widehat{Exp}^{\perp}: \widehat{\nu}M \rightarrow S^{N}$, which are open and dense in $S^{N}$. \\  

As explained in the discussion before Proposition \ref{spheretcmeaning}, integrating $\vert det(dExp^{\perp}) \vert$ over neighborhoods of $\nu^{<\pi}M$ defines a positive measure on $\nu^{<\pi}M$ which is absolutely continuous with respect to  $dVol_{\nu^{<\pi}M}$ and is, in a natural sense, the pull-back of the measure on $S^{N}$ via $Exp^{\perp}$.  We will denote this measure by $d\mu$.  If $\phi$ is a measurable function on $\nu^{<\pi}M$, the integral of $\phi$ with respect to this measure is given by integrating against $\vert det(dExp^{\perp}) \vert dVol_{\nu^{<\pi}M}$: 

\medskip
\begin{center}
%\begin{equation}
$\displaystyle \text{\LARGE $\int\limits_{\text{\scriptsize $\nu^{<\pi}M$}}$} \phi \text{   d$\mu$} = \text{\LARGE $\int\limits_{\text{\scriptsize $\nu^{<\pi}M$}}$} \phi \vert det(dExp^{\perp}) \vert dVol_{\nu^{<\pi}M}.$
%\end{equation}
\end{center}
\medskip   

For any regular point $\vec{v}$ of $Exp^{\perp}:\nu^{<\pi}M \rightarrow S^{N}$, $Exp^{\perp}$ is a local diffeomorphism from a neighborhood $V$ of $\vec{v}$ to a neighborhood $Q$ of $Exp^{\perp}(\vec{v})$ in $S^{N}_{reg}$.  We then have:   

\begin{center}
\begin{equation}
\label{spherelocalcov}
\displaystyle \text{\LARGE $\int\limits_{\text{\scriptsize $V$}}$} \phi d\mu = \text{\LARGE $\int\limits_{\text{\scriptsize $V$}}$} \phi \vert det(dExp^{\perp}) \vert dVol_{\nu^{<\pi}M} = \text{\LARGE $\int\limits_{\text{\scriptsize $Q$}}$} \phi \circ (Exp^{\perp})^{-1} dVol_{S^{N}}.
\end{equation}
\end{center}

This implies that the pre-image of $S^{N}_{reg}$ in $\nu^{<\pi}M$ is of full measure relative to $d \mu$:  If $\vec{v}$ is a regular point of $Exp^{\perp}$ whose image in $S^{N}$ is a critical value, then let $V$ and $Q$ be neighborhoods of $\vec{v}$ and $Exp^{\perp}(\vec{v})$ with $Exp^{\perp}: V \rightarrow Q$ a diffeomorphism as in (\ref{spherelocalcov}).  $Q \cap S^{N}_{reg}$ contains a set which is open, dense and of full measure in $Q$, so $(Exp^{\perp})^{-1}(S^{N}_{reg}) \cap V$ likewise contains a set which is open, dense and of full measure in $V$.  This implies that: \\  

\begin{center}  
$\displaystyle \text{\LARGE $\int\limits_{\text{\scriptsize $(Exp^{\perp})^{-1}(S^{N} \setminus S^{N}_{reg}) \cap V$}}$}\vert det(dExp^{\perp}) \vert dVol_{\nu^{<\pi}M} = 0.$
\end{center}

\medskip

We can cover the regular points in the pre-image of $S^{N} \setminus S^{N}_{reg}$ with open sets $V$ as above.  The integral of $|det(dExp^{\perp})|$ over the critical points of $Exp^{\perp}$ is zero, and this implies that $(Exp^{\perp})^{-1}(S^{N}_{reg})$ is of full measure for $d \mu$. \\  

We now apply the standard fact from Morse theory that if $f$ is a Morse function on a closed manifold $M$, then $\beta_{i}(M;F)$, the $i^{th}$ Betti number of $M$ with coefficients in the field $F$, is bounded above by $C_{i}(f)$, the number of critical points of $f$ which have index $i$.  For $q$ in $S^{N}_{reg}$, the distance function $d_{q}$ is Morse on $M$, and its critical points are in $1-1$ correspondence with the pre-images of $q$ via $Exp^{\perp}$, as explained in Proposition \ref{spherefullmeasure} and Lemma \ref{hesslemma}.  Letting $C_{i}(d_{q})$ denote the number of critical points of $d_{q}$ of index $i$, and letting $\sharp(Exp^{\perp})^{-1}(q)$ denote the pre-image count via $Exp^{\perp}$ for a point $q$ of $S^{N}$, we therefore have:  

\bigskip
%\begin{center}

$\mathcal{T}(M) = \displaystyle \text{\scriptsize $\frac{1}{Vol(S^{N})}$} \text{\LARGE $\int\limits_{\text{\footnotesize $\nu^{<\pi}M$}} \vert$} det(dExp^{\perp}_{\vec{v}}) \text{\LARGE $\vert$} dVol_{\text{\tiny $\nu^{<\pi}M$}}(\vec{v}) = \text{\scriptsize $\frac{1}{Vol(S^{N})}$} \text{\LARGE $\int\limits_{\text{\footnotesize $(Exp^{\perp})^{-1}(S^{N}_{reg})$}} \vert$} det(dExp^{\perp}_{\vec{v}}) \text{\LARGE $\vert$} dVol_{\text{\tiny $\nu^{<\pi}M$}}(\vec{v})$

\bigskip  

$ = \displaystyle \text{\scriptsize $\frac{1}{Vol(S^{N})}$} \text{\LARGE $\int\limits_{\text{\scriptsize $S^{N}_{reg}$}}$} \sharp(Exp^{\perp})^{-1}(q) dVol_{S^{N}}(q) = \text{\scriptsize $\frac{1}{Vol(S^{N})}$} \text{\LARGE $\int\limits_{\text{\scriptsize $S^{N}_{reg}$}}$} \left( \sum\limits_{i=0}^{n}C_{i}(d_{q}) \right) dVol_{S^{N}}(q) \geq \sum\limits_{i=0}^{n} \beta_{i}(M;F)$

%\end{center}
\medskip  

The first equality above follows from Proposition \ref{spheretcmeaning} and the second from the fact that $(Exp^{\perp})^{-1}(S^{N}_{reg})$ is of full measure for $d\mu$.  The third follows from the change-of-variables formula (\ref{spherelocalcov}) (with $\phi = 1$), and the fourth from the correspondence between the pre-images of $q$ via $Exp^{\perp}$ and the critical points of the distance function $d_{q}$ on $M$, as in Lemma \ref{hesslemma}.  We have also used the fact that $S^{N}_{reg}$ is of full measure in $S^{N}$. \\  

This establishes Theorem \ref{spherecl}.A, that $\sum\limits_{i=0}^{n} \beta_{i}(M;F) \leq \mathcal{T}(M)$. \\

If $\mathcal{T}(M) < 3$, there must be a set $\mathcal{S}$ within $S^{N}$ of positive measure, which must therefore intersect $S^{N}_{reg}$, for which the pre-image count $\sharp(Exp^{\perp})^{-1}(q)$ is less than $3$.  For any $q_{0} \in \mathcal{S} \cap S^{N}_{reg}$ the pre-image count must therefore be equal to $2$, corresponding to the critical points of $d_{q_{0}}$ at its global minimum and maximum.  Reeb proved in ~\cite{Re} that if a closed manifold admits a Morse function with its global minimum and maximum as its only critical points, it is homeomorphic to a sphere (although it may not be diffeomorphic to the standard sphere.)  For $q_{0}$ as above, $d_{q_{0}}$ provides such a Morse function, and this completes the proof of Theorem \ref{spherecl}.B.

\end{proof}  

For odd-dimensional submanifolds of spheres, there is the following stronger version of Theorem \ref{spherecl}.B.  This will be important in proving Theorem \ref{cpcl}.B:     

\begin{theorem}
\label{strongerspherecl2}

Let $M^{2m+1}$ be an odd-dimensional compact manifold, isometrically immersed in the sphere $S^{N}$, with $\mathcal{T}(M) < 4$. \\  

Then $M$ is homeomorphic to the sphere $S^{2m+1}$.  

\end{theorem}

\begin{proof}

By the proofs of Theorem \ref{spherecl}.A, Proposition \ref{spherefullmeasure} and Lemma \ref{hesslemma}, if $\mathcal{T}(M) < 4$, there is a point $q_{0}$ in $S^{2N+1}$ such that the distance function $d_{q_{0}}$ is a Morse function on $M$ with fewer than $4$ critical points.  $d_{q_{0}}$ must have a critical point of index $0$ at a point $p_{1}$ of $M$ where it realizes its global minimum, and a critical point of index $2m+1$ at another point $p_{2}$ where it attains its global maximum.  If $d_{q_{0}}$ had a third critical point $p_{3}$ of index $k$, it would have a fourth critical point $p_{4}$ of index $2m+1-k$ because $\chi(M) = 0$.  This is impossible, so $p_{1}$ and $p_{2}$ are in fact the only critical points of $d_{q_{0}}$ on $M$, and $M$ is homeomorphic to $S^{2m+1}$ by Reeb's theorem, \cite{Re}. 

\end{proof}

A similar statement holds for odd-dimensional manifolds immersed in Euclidean spaces. \\  

We will prove Part C of Theorem \ref{spherecl} as a consequence of the following proposition.  Parts A and B of Theorem \ref{spherecl} also follow from this proposition, however many of the observations in the proofs of Theorems \ref{spherecl}.A and \ref{spherecl}.B given above will be important in explaining and proving our results for complex projective manifolds.   

\begin{proposition}
\label{cltcequalsspheretc}

For $M$ isometrically immersed in $S^{N}$ and $S^{N}$ embedded as the unit sphere in $\mathbb{R}^{N+1}$, \\ 

\begin{center} 
$\mathcal{T}_{S^{N}}(M) = \mathcal{T}_{\R^{N+1}}(M)$.  
\end{center} 

\medskip 

\end{proposition}

\begin{proof}%[Proof of ~\ref{cltcequalsspheretc}]

Let $\nu^{1}M$ be the unit normal bundle of $M$ in $S^{N}$, and {\Small $\widetilde{\nu^{1}M}$} its unit normal bundle in $\R^{N+1}$.  Let $\vec{\nu}$ be the outward unit normal vector to $S^{N}$ in $\R^{N+1}$.  Any unit normal vector $\vec{z}$ to $M$ in $\R^{N+1}$ is of the form $\cos(\theta) \vec{\nu} + \sin(\theta) \vec{u}$, for $\vec{u}$ a unit normal vector to $M$ in $S^{N}$ and $\theta \in [0,\pi]$.  Let $\widetilde{A_{\vec{u}, \theta}}$ be the second fundamental form of $M$ in $\R^{N+1}$ for a normal vector $\cos(\theta) \vec{\nu} + \sin(\theta) \vec{u}$ as above, and let $A_{\vec{u}}$ be the second fundamental form of $M$ in $S_{N}$ for the corresponding $\vec{u}$.  For any principal vector $e$ of $A_{\vec{u}}$ with principal curvature $\kappa$, we have the following identity :  

\begin{center} 

\medskip

$\widetilde{A_{\vec{u}, \theta}}(e) = -\text{\LARGE $($} \nabla_{e}^{\R^{N+1}} \text{\Large $($} \cos(\theta) \vec{\nu}+\sin(\theta) \vec{u} \text{\Large $)$} \text{\LARGE $)$}^{\top} = -\text{\LARGE $($}\cos(\theta) \nabla_{e}^{\R^{N+1}}\vec{\nu}+\sin(\theta) \nabla_{e}^{\R^{N+1}}\vec{u} \text{\LARGE $)$}^{\top}$ 

\bigskip  

$= -\cos(\theta) e + \sin(\theta) (-\nabla_{e}^{S^{N}}\vec{u})^{\top} = \left( \kappa\sin(\theta)  - \cos(\theta) \right) e.$

\bigskip  

\end{center} 

$e$ is therefore an eigenvector of $\widetilde{A_{\vec{u}, \theta}}$ with eigenvalue $\kappa\sin(\theta) - \cos(\theta)$.  Then we have: \\

\begin{center} 

$\mathcal{T}_{S^{N}}(M) = \frac{1}{Vol(S^{N})} \displaystyle \text{\LARGE $\int\limits_{\text{\scriptsize $\nu^{1}M$}} \int\limits_{\text{\scriptsize$0$}}^{\text{\scriptsize $\pi$}}$} \text{\LARGE $\vert$} \prod\limits_{i=1}^{n} \left( \kappa_{i}\sin(\theta) - \cos(\theta) \right) \text{\LARGE $\vert$} \sin^{(N-n-1)}(\theta)$ $d\theta$ $dVol_{\nu^{1}M}(\vec{u})$  

\bigskip  

$= \frac{1}{Vol(S^{N})} \displaystyle \text{\LARGE $\int\limits_{\text{\scriptsize $\nu^{1}M$}} \int\limits_{\text{\scriptsize$0$}}^{\text{\scriptsize $\pi$}}$} \text{\LARGE $\vert$} det(\widetilde{A_{\vec{u}, \theta}}) \text{\LARGE $\vert$}  \sin^{(N-n-1)}(\theta)$ $d\theta$ $dVol_{\nu^{1}M}(\vec{u})$

\bigskip  

$= \frac{1}{Vol(S^{N})} \displaystyle \text{\LARGE $\int\limits_{\text{\scriptsize $\widetilde{\nu^{1}M}$}}$} \text{\LARGE $\vert$} det(\widetilde{A_{\vec{z}}}) \text{\LARGE $\vert$} dVol_{\widetilde{\nu^{1}M}}(\vec{z}) =  \mathcal{T}_{\R^{N+1}}(M). $

\end{center}

\end{proof}

This proof has a simple geometric meaning: \\ 

The image of {\small $\nu_{p}^{<\pi}M$} under the normal exponential map, as a subset of $S^{N}$, is the same as the image of {\small $\widetilde{\nu^{1}_{p}M}$} under the Gauss map.  In fact, these maps coincide under the natural identification between $\nu^{<\pi}M$ and {\small $\widetilde{\nu^{1}M}$}.  Up to the normalization by $Vol(S^{N})$, $\mathcal{T}_{S^{N}}(M)$ is the mass of $\nu^{<\pi}M$ with the measure pulled back from $S^{N}$ by the normal exponential map, as in the proof of Proposition \ref{spheretcmeaning} and Theorems \ref{spherecl}.A and \ref{spherecl}.B.  Similarly, $\mathcal{T}_{\R^{N+1}}(M)$ is the total mass of the unit normal bundle to $M$ in $\R^{N+1}$, with the positive measure pulled back from $S^{N}$ by the Gauss map.  Because these maps coincide, the pulled-back measures coincide.  

\begin{proof}[Proof of Part C of Theorem ~\ref{spherecl}]

Let $M^{n}$ be a submanifold of $S^{N}$ with $\mathcal{T}_{S^{N}}(M) = 2$.  By isometrically embedding $S^{N}$ in $\R^{N+1}$ as above, $M$ is also a submanifold of $\R^{N+1}$ with $\mathcal{T}_{\R^{N+1}}(M) = 2$, and by Part C of Theorem \ref{cl}, $M$ is the boundary of a convex set in an affine subspace $\mathcal{A}$ of $\R^{N+1}$ of dimension $(n+1)$.  $M$ is therefore a closed, embedded $n$-dimensional submanifold of $\mathcal{A} \cap S^{N}$, which is homeomorphic to $S^{n}$.  $M$ must therefore be equal to $\mathcal{A} \cap S^{N}$.  Let $\mathcal{V}$ be the unique $(n+2)$-dimensional linear subspace of $\R^{N+1}$ containing the affine subspace $\mathcal{A}$.  $\mathcal{V} \cap S^{N}$ is an $(n+1)$-dimensional totally geodesic subsphere of $S^{N}$, and $M$ is embedded in $\mathcal{V} \cap S^{N}$ as the boundary of a geodesic ball. \\

To see the converse, if $S^{n}$ is the boundary of a geodesic ball in a totally geodesic $(n+1)$-dimensional subsphere $S^{n+1}$ in $S^{N}$, we isometrically embed $S^{n+1}$ as the unit sphere in $\R^{n+2}$.  This isometric embedding takes $S^{n}$ to the boundary of a geodesic ball about a point $q$ in the unit sphere in $\R^{n+2}$ - the boundary of such a geodesic ball is given by the intersection of the unit sphere $S^{n+1}$ with an $(n+1)$-dimensional affine subspace $\mathcal{A}$ of $\R^{n+2}$.  Therefore, by Proposition \ref{cltcequalsspheretc} and Theorem \ref{cl}.C, $\mathcal{T}_{S^{n+1}}(S^{n}) = 2$, and because the total absolute curvature of submanifolds of spheres is preserved under totally geodesic embeddings into higher-dimensional spheres, $\mathcal{T}_{S^{N}}(S^{n}) = \mathcal{T}_{S^{n+1}}(S^{n})$.  

\end{proof} 

We end this section by noting that any compact Riemannian manifold can be isometrically embedded in a sphere of constant curvature $1$ of sufficiently high dimension:  The Nash embedding theorems, in \cite{NaI} and \cite{NaII}, imply that every compact Riemannian manifold $M$ can be isometrically embedded in a Euclidean space $\R^{k}$ of sufficiently high dimension.  We can take such an embedding to have its image in a cell $\lbrack 0, D \rbrack^{k}$.  Letting $d$ be an integer greater than $D^{2}$, the long diagonal of the unit cube in $\R^{d}$ has length greater than $D$, so $\lbrack 0, D \rbrack^{k}$, and $M$ itself, can be isometrically embedded in the unit cube in $\R^{dk}$.  Finally, the unit cube in $\R^{dk}$ can be isometrically embedded in a fundamental domain for a flat torus in $S^{2dk-1}$, so $M$ also admits such an embedding. 

%%%%%%%%%%%%%%%%%%%%%%%%%%%%%%%%%%%%%%%%%%%%%%%%%%%%%%%%%%%%%%%%%%%%%%%%%%%%%%%%%%%%%%%%%%%%%%%%%%%%%%%%%%%%%%%%%%%%%%%%%

\section{Total Curvature and the Betti Numbers of Complex Projective Manifolds}
\label{bettis}

We recall the definition of the total absolute curvature of a complex projective manifold, in Definition \ref{cptc}, as an integral in terms of its second fundamental form:  If $M$ is of complex dimension $m$, holomorphically immersed in $\C P^{N}$, then: \\ 

%\begin{center}

$\mathcal{T}(M) = \frac{2}{Vol(\C P^{N})} \displaystyle \text{\LARGE $\int\limits_{\text{\scriptsize $\nu^{<\frac{\pi}{2}}M$}}$}  \text{\LARGE $\vert$} det \left( \cos(r) Id_{T_{p}M} -\left( \frac{\sin r}{r} \right) A_{\vec{v}} \right) \text{\LARGE $\vert$} \cos(r) \left( \frac{\sin r}{r} \right)^{\tiny (2N-2m-1)} dVol_{ \text{\tiny $\nu^{<\frac{\pi}{2}}M$}}(\vec{v}).  $ 

%\end{center}

\medskip 

It will be helpful to note that this can also be written as: \\ 

\begin{center}

$\mathcal{T}(M) = \frac{2}{Vol(\C P^{N})} \displaystyle \text{\LARGE $\int\limits_{\text{\scriptsize $\nu^{1}M$}} \int\limits_{\text{\scriptsize $0$}}^{\text{\scriptsize $\frac{\pi}{2}$}}$ }\text{\LARGE $\vert$} det \left( \cos(r) Id_{T_{p}M} -\sin(r) A_{\vec{u}} \right) \text{\LARGE $\vert$} \cos(r) \sin^{(2N-2m-1)}(r)$ $dr$ $dVol_{\text{\tiny $\nu^{1}M$}}(\vec{u}).$

\end{center}
\medskip

One of the facts we will need to prove Theorem \ref{cpcl} is:  

\begin{proposition}
\label{tcsame}

Let $M$ be a compact complex manifold holomorphically immersed in $\C P^{N}$, and let $\widetilde{M}$ be the circle bundle over $M$ which is induced by the Hopf fibration $\pi : S^{2N+1} \rightarrow \C P^{N}$.  With the immersion of $\widetilde{M}$ in $S^{2N+1}$ induced by the immersion of $M$ in $\C P^{N}$, we have:  

\bigskip 
\begin{center}
$\mathcal{T}_{S^{2N+1}}(\widetilde{M}) = \mathcal{T}_{\C P^{N}}(M).$
\end{center}

\end{proposition}
\medskip

Proposition \ref{cppc} implies that the total absolute curvature of a complex projective manifold has the same meaning as the total absolute curvature of a submanifold of Euclidean space.  However, the adaptation of Chern and Lashof's proofs which we applied to submanifolds of spheres in Section \ref{sphereresults} breaks down for submanifolds of complex projective space.  This is because the cut locus of a point in complex projective space is a complex projective hyperplane, of real codimension $2$.  If $M$ is a submanifold of complex projective space which meets the cut locus of a point $q$, the distance function from $q$ may not be smooth when restricted to $M$ - in particular, it may not be Morse.  If $M$ has real dimension $2$ or greater, the union over the points of $M$ of their cut loci in $\C P^{N}$ has positive measure.  Because of this, the total absolute curvature is no longer equal to the average number of critical points of a family of Morse functions on $M$. \\

For closed complex submanifolds of $\C P^{N}$, we can be more precise about the scope of this issue: such a manifold necessarily meets the cut locus of every point in $\C P^{N}$, because a closed complex projective manifold of complex dimension $m$ in $\C P^{N}$ intersects every linear subspace of dimension $N-m$ or greater.  Because of this, in general, none of the ambient distance functions will have smooth restrictions to $M$ (one exception is among complex projective manifolds contained linear subspaces, in which case some will be constant on $M$.)  In addition, the intersection of a closed $m$-dimensional complex projective manifold with a linear subspace $\C P^{N-1}$ is a subvariety of $M$ of dimension $m-1$.  This intersection is topologically essential in $M$ - for example, the Lefschetz hyperplane theorem implies that the homomorphisms on homology induced by the inclusion of a hyperplane section are surjective up to one less than half the real dimension of $M$.  This implies that, even if one smooths the distance functions from points in the ambient projective space, their critical sets on closed complex projective manifolds will include topologically essential submanifolds, and they will therefore not be Morse. \\

Because the Chern-Lashof theorems hold for submanifolds of spheres in full generality, as shown in Section \ref{sphereresults}, we will prove Theorem \ref{cpcl} by relating the geometry and topology of a complex projective manifold to those of its pre-image via the Hopf fibration.  Proposition \ref{tcsame} is the first observation that we will need to do this - its proof is based on the following result, which gives a complete description of the second fundamental form of $\widetilde{M}$ in $S^{2N+1}$.    

\begin{proposition}  
\label{IIisthesame}

Let $M$ be a complex manifold holomorphically immersed in $\C P^{N}$ and $\widetilde{M}$ its pre-image in $S^{2N+1}$ via the Hopf fibration.  Let $\vec{u}$ be a unit normal vector to $M$ at $p$; $e_{1}, ... , e_{n}$ a set of principal vectors for the second fundamental form $A_{\vec{u}}$; and $\kappa_{1}, ... , \kappa_{n}$ the principal curvatures of $e_{1}, ... , e_{n}$.  Let $\widetilde{u}, \widetilde{e}_{1}, ... , \widetilde{e}_{n}$ be the horizontal lifts of $\vec{u}, e_{1}, ..., e_{n}$ at any point $\widetilde{p}$ in $\widetilde{M}$ above $p$. \\ 

Then $\widetilde{u}$ is normal to $\widetilde{M}$, $\widetilde{e}_{1}, ... , \widetilde{e}_{n}$ are principal vectors for the second fundamental form $A_{\widetilde{u}}$, and each principal direction $\widetilde{e}_{i}$ has the same principal curvature $\kappa_{i}$ as its image $e_{i}$.  The tangent to the Hopf fibre through $\widetilde{p}$ is also a principal direction for $A_{\widetilde{u}}$ with principal curvature $0$.

\end{proposition}

\begin{proof}  

We let $\widetilde{\nu}$ denote the outward unit normal to $S^{2N+1}$ in $\R^{2N+2} = \C^{N+1}$.  We let $J$ denote the complex structure of $\C P^{N}$ and of $M$, and also of $\C^{N+1}$.  We let $ h_{FS}$ denote the canonical metric on $\C P^{N}$, $h$ the induced metric on $M$ and $\widetilde{h}$ the canonical metric on $\C^{N+1}$ and on the unit sphere $S^{2N+1}$ in $\C^{N+1}$.  We let $\widetilde{E_{0}}$ denote the vector field $J(\widetilde{\nu})$ on $S^{2N+1}$.  $\widetilde{E_{0}}$ is then a unit-length vector field on $S^{2N+1}$ tangent to the Hopf fibres.  Its orthogonal complement in the tangent bundle of $S^{2N+1}$ is invariant under the action of the complex structure of $\C^{N+1}$, and the action of this complex structure on this distribution in $TS^{2N+1}$ induces the complex structure on $\C P^{N}$. \\  

For $e_{i}$ a principal vector for $\vec{u}$ with principal curvature $\kappa_{i}$ as above, let $E_{i}$ and $U$ be unit-length vector fields on a neighborhood of $\C P^{N}$ which extend $e_{i}$ and $\vec{u}$ respectively, with $E_{i}$ tangent and $U$ normal to $M$.  For any vector or vector field defined on a neighborhood of $\C P^{N}$, let a tilde denote its horizontal lift to any neighborhood in $S^{2N+1}$ via the Hopf fibration. \\

By O'Neill's formula,   

\begin{center}
\begin{equation}
\label{horzlift}
\nabla_{\widetilde{E}_{i}}^{S^{2N+1}}\widetilde{U} = \widetilde{\nabla_{E_{i}}^{\C P^{N}}U} + \frac{1}{2} \lbrack \widetilde{E}_{i}, \widetilde{U} \rbrack^{v}. 
\end{equation} 
\end{center}

\bigskip 

Because $\nabla_{E_{i}}^{\C P^{N}}U = -A_{\vec{u}}(e_{i}) + {\huge (} \nabla_{E_{i}}^{\C P^{N}}U {\huge )}^{\perp} = -\kappa_{i}e_{i} +  {\huge (} \nabla_{E_{i}}^{\C P^{N}}U {\huge )}^{\perp}$ at $p$, where ${\huge (} \nabla_{E_{i}}^{\C P^{N}}U {\huge )}^{\perp}$ is the component of $\nabla_{E_{i}}^{\C P^{N}}U$ normal to $M$, and because the horizontal lift of $-\kappa_{i}e_{i} +  {\huge (} \nabla_{E_{i}}^{\C P^{N}}U {\huge )}^{\perp}$ is given by $-\kappa_{i} \widetilde{e}_{i} + \widetilde{ {\huge (} \nabla_{E_{i}}^{\C P^{N}}U {\huge )}^{\perp} } $, at $\widetilde{p}$ we have:  

\begin{center}
\begin{equation}
\label{otherhl}
\nabla_{\widetilde{E}_{i}}^{S^{2N+1}}\widetilde{U} = -\kappa_{i} \widetilde{e}_{i} + \widetilde{ {\huge (} \nabla_{E_{i}}^{\C P^{N}}U {\huge )}^{\perp} } + \frac{1}{2} \lbrack \widetilde{E}_{i}, \widetilde{U} \rbrack^{v}. 
\end{equation}
\end{center}
\medskip  

We note that $\widetilde{ {\huge (} \nabla_{E_{i}}^{\C P^{N}}U {\huge )}^{\perp} }$ is normal to $\widetilde{M}$.  Noting that vertical directions for the Hopf fibration are tangent to $\widetilde{M}$, we can rewrite (\ref{horzlift}) and (\ref{otherhl}) as follows: \\ 

%\bigskip  
\begin{center}
$\nabla_{\widetilde{E}_{i}}^{S^{2N+1}}\widetilde{U} = -\kappa_{i} \widetilde{e}_{i} + (\frac{1}{2}) \lbrack \widetilde{E}_{i}, \widetilde{U} \rbrack^{v}  + {\huge (} \nabla_{\widetilde{E}_{i}}^{S^{2N+1}}\widetilde{U} {\huge )}^{\perp}$,
\end{center}
\bigskip  

where ${\huge (} \nabla_{\widetilde{E}_{i}}^{S^{2N+1}}\widetilde{U} {\huge )}^{\perp}$ again denotes the normal component of $\nabla_{\widetilde{E}_{i}}^{S^{2N+1}}\widetilde{U}$, which is equal to $\widetilde{ {\huge (} \nabla_{E_{i}}^{\C P^{N}}U {\huge )}^{\perp} }$.  Noting also that $\nabla_{\widetilde{E}_{i}}^{S^{2N+1}}\widetilde{U} = -A_{\widetilde{u}}(\widetilde{e}_{i}) +{\huge (} \nabla_{\widetilde{E}_{i}}^{S^{2N+1}}\widetilde{U} {\huge )}^{\perp}$ at $\widetilde{p}$, where $A_{\widetilde{u}}$ is the second fundamental form of $\widetilde{M}$ in the normal direction $\widetilde{u}$, we infer that: \\  

\begin{center}
$A_{\widetilde{u}}(\widetilde{e}_{i}) = \kappa_{i} \widetilde{e}_{i} - \frac{1}{2} \lbrack \widetilde{E}_{i}, \widetilde{U} \rbrack^{v}.$
\end{center}
\bigskip

The proof that $\widetilde{e}_{i}$ is a principal vector for $\widetilde{u}$, with principal curvature $\kappa_{i}$, will be complete once we show that $\lbrack \widetilde{E}_{i}, \widetilde{U} \rbrack^{v}$ is zero. \\ 

Because the vertical distribution for the Hopf fibration is spanned by $\widetilde{E}_{0}$, $\lbrack \widetilde{E}_{i}, \widetilde{U} \rbrack^{v}$ is equal to $\widetilde{h} {\Huge (} \lbrack \widetilde{E}_{i}, \widetilde{U} \rbrack, \widetilde{E}_{0} {\Huge )}\widetilde{E}_{0}$.  We write this in terms of the connection in $S^{2N+1}$ as follows: 

\begin{center}
\begin{equation}
\label{liebracket}
\widetilde{h} {\Huge (} \lbrack \widetilde{E}_{i}, \widetilde{U} \rbrack, \widetilde{E}_{0} {\Huge )}\widetilde{E}_{0} = \text{\LARGE $($} \widetilde{h} \text{\Large $($} \nabla_{\widetilde{E}_{i}}^{S^{2N+1}}\widetilde{U},  \widetilde{E}_{0} \text{\Large $)$} -  \widetilde{h} \text{\Large $($} \nabla_{\widetilde{U}}^{S^{2N+1}}\widetilde{E}_{i}, \widetilde{E}_{0} \text{\Large $)$} \text{\LARGE $)$} \widetilde{E}_{0}. 
%\widetilde{h} {\Huge (} \lbrack \widetilde{E}_{i}, \widetilde{U} \rbrack, \widetilde{E}_{0} {\Huge )}\widetilde{E}_{0} = \widetilde{h} {\Huge (} \nabla_{\widetilde{E}_{i}}^{S^{2N+1}}\widetilde{U} - \nabla_{\widetilde{U}}^{S^{2N+1}}\widetilde{E}_{i}, \widetilde{E}_{0} {\Huge )}\widetilde{E}_{0}= \text{\LARGE $($} \widetilde{h} \text{\Large $($} \nabla_{\widetilde{E}_{i}}^{S^{2N+1}}\widetilde{U},  \widetilde{E}_{0} \text{\Large $)$} -  \widetilde{h} \text{\Large $($} \nabla_{\widetilde{U}}^{S^{2N+1}}\widetilde{E}_{i}, \widetilde{E}_{0} \text{\Large $)$} \text{\LARGE $)$} \widetilde{E}_{0}. 
\end{equation}
\end{center}
\medskip 

Because the Euclidean metric on $\C^{N+1}$ is Hermitian and $\widetilde{E}_{0} = J(\widetilde{\nu})$, we have: \\ 

\begin{center}
$\widetilde{h}{\Huge(} \nabla_{\widetilde{E}_{i}}^{S^{2N+1}}\widetilde{U}, \widetilde{E}_{0} {\Huge )} = \widetilde{h}{\Huge(} \nabla_{\widetilde{E}_{i}}^{\C^{N+1}}\widetilde{U}, \widetilde{E}_{0} {\Huge )} = \widetilde{h}{\Huge(} J(\nabla_{\widetilde{E}_{i}}^{\C^{N+1}}\widetilde{U}), J(\widetilde{E}_{0}) {\Huge )} = -\widetilde{h}{\Huge(} J(\nabla_{\widetilde{E}_{i}}^{\C^{N+1}}\widetilde{U}), \widetilde{\nu} {\Huge )}. $
\end{center}

\medskip  

Because the metric on $\C^{N+1}$ is K\"ahler, its connection $\nabla^{\C^{N+1}}$ commutes with the complex structure $J$, so letting $J(\widetilde{U})$ denote the vector field which results from applying the complex structure of $\C^{N+1}$ to $\widetilde{U}$, we have: \\ 

\begin{center} 

$\widetilde{h}{\Huge(} J(\nabla_{\widetilde{E}_{i}}^{\C^{N+1}}\widetilde{U}), \widetilde{\nu} {\Huge )} = \widetilde{h}{\Huge(} \nabla_{\widetilde{E}_{i}}^{\C^{N+1}}J(\widetilde{U}), \widetilde{\nu} {\Huge )} = \widetilde{E}_{i}{\Huge (}\widetilde{h}{\huge (}J(\widetilde{U}), \widetilde{\nu} {\huge )} {\Huge )} - \widetilde{h} {\Huge (} J(\widetilde{U}), \nabla_{\widetilde{E}_{i}}^{\C^{N+1}}\widetilde{\nu} {\Huge )}.$

\end{center}
\medskip  

Because $\widetilde{\nu}$ is the outward unit normal to $S^{2N+1}$, we have $\nabla_{\widetilde{E}_{i}}^{\C^{N+1}}\widetilde{\nu} = \widetilde{E}_{i}$, so the above is equal to: \\ 

\begin{center}
$\widetilde{E}_{i}{\Huge (}\widetilde{h}{\huge (}J(\widetilde{U}), \widetilde{\nu} {\huge )} {\Huge )} - \widetilde{h} {\Huge (} J(\widetilde{U}), \widetilde{E}_{i} {\Huge )}.$
\end{center}
\medskip 

Because $M$ is a complex submanifold of $\C P^{N}$, its tangent and normal spaces are preserved by the complex structure of $\C P^{N}$.  Letting $J(U)$ denote the vector field which results from applying the complex structure of $\C P^{N}$ to $U$, we therefore have that $J(U)$ is normal to $M$.  Because the complex structure on $\C^{N+1}$ preserves the subbundle of $TS^{2N+1}$ orthogonal to $\widetilde{E}_{0}$, $J(\widetilde{U})$ is a horizontal vector field for the Hopf fibration.   And because the complex structure on $\C P^{N}$ is induced by the action of the complex structure of $\C^{N+1}$ as described above, $J(\widetilde{U})$ is the horizontal lift of $J(U)$. \\ % - $J(\widetilde{U})$ is the horizontal lift of $J(U)$, with our notational conventions, $J(\widetilde{U}) = \widetilde{J(U)}$.  

Applying this to the formula above, we have that $\widetilde{h} {\Huge (} J(\widetilde{U}), \widetilde{E}_{i} {\Huge )} = h_{FS} {\Huge (} J(U), E_{i} {\Huge )} = 0$ because $J(U)$ is normal to $M$ and $E_{i}$ is tangent to $M$ in $\C P^{N}$.  We have $\widetilde{h} {\Huge (} J(\widetilde{U}), \widetilde{\nu} {\Huge )} \equiv 0$ because $J(\widetilde{U})$ is tangent and $\widetilde{\nu}$ normal to $S^{2N+1}$.  This then implies that $\widetilde{E}_{i}{\Huge (}\widetilde{h}{\huge (}J(\widetilde{U}), \widetilde{\nu} {\huge )} {\Huge )}$ is zero.  And this implies that the expression $\widetilde{h}{\Huge(} \nabla_{\widetilde{E}_{i}}^{S^{2N+1}}\widetilde{U}, \widetilde{E}_{0} {\Huge )}$ in (\ref{liebracket}) is zero. \\   

By a similar set of observations, the term $\widetilde{h} {\Huge (} \nabla_{\widetilde{U}}^{S^{2N+1}}\widetilde{E}_{i}, \widetilde{E}_{0} {\Huge )}$ in (\ref{liebracket}) is also zero.  This implies that $\lbrack \widetilde{E}_{i}, \widetilde{U} \rbrack^{v}$ is zero and completes the proof that $\widetilde{e}_{i}$ is a principal vector for $A_{\widetilde{u}}$, with principal curvature $\kappa_{i}$, just as $e_{i}$ is a principal vector for $A_{\vec{u}}$ with prinicipal curvature $\kappa_{i}$.  To see that the tangent to the Hopf fibre is also a principal vector, with principal curvature zero, we note that: \\   

\begin{center}
$\widetilde{h} {\Huge (} \nabla_{\widetilde{E}_{0}}^{S^{2N+1}} \widetilde{U}, \widetilde{E}_{0} {\Huge )} =  \widetilde{E_{0}}( \widetilde{h} {\Huge (} \widetilde{U}, \widetilde{E_{0}} {\Huge )} ) - \widetilde{h} {\Huge (} \widetilde{U}, \nabla^{S^{2N+1}}_{\widetilde{E_{0}}}\widetilde{E}_{0} {\Huge )}.$
\end{center} 

\medskip  

Because the Hopf fibres are geodesics of $S^{2N+1}$, we have $\nabla^{S^{2N+1}}_{\widetilde{E}_{0}}\widetilde{E}_{0} = 0$.  We also have $\widetilde{h} {\Huge (} \widetilde{U}, \widetilde{E}_{0} {\Huge )} \equiv 0$ because $\widetilde{U}$ is normal and $\widetilde{E}_{0}$ tangent to $\widetilde{M}$.  This implies that $\widetilde{E_{0}}( \widetilde{h} {\Huge (} \widetilde{U}, \widetilde{E_{0}} {\Huge )} ) = 0$, and as a consequence, $\widetilde{h} {\Huge (} \nabla_{\widetilde{E}_{0}}^{S^{2N+1}} \widetilde{U}, \widetilde{E}_{0} {\Huge )} = 0$. \\

Letting $\widetilde{E}_{i}$ be a horizontal lift of a vector field $E_{i}$ as above for $i = 1, 2, \cdots, n$, we have that $\widetilde{h} {\Huge (} \nabla_{\widetilde{E}_{0}}^{S^{2N+1}} \widetilde{U}, \widetilde{E}_{i} {\Huge )} = \widetilde{E}_{0} (\widetilde{h} {\Huge (} \widetilde{U}, \widetilde{E}_{i} {\Huge )}) - \widetilde{h} {\Huge (} \widetilde{U}, \nabla^{S^{2N+1}}_{\widetilde{E_{0}}}\widetilde{E_{i}} {\Huge )} $.  $\widetilde{E}_{0} (\widetilde{h} {\Huge (} \widetilde{U}, \widetilde{E}_{i} {\Huge )})$ is zero because $\widetilde{h} {\Huge (} \widetilde{U}, \widetilde{E}_{i} {\Huge )}$ is zero, so we are left with: \\ 
 
\begin{center}
$ - \widetilde{h} {\Huge (} \widetilde{U}, \nabla^{S^{2N+1}}_{\widetilde{E_{0}}}\widetilde{E_{i}} {\Huge )}  = \widetilde{h} {\Huge (} \widetilde{U}, \lbrack \widetilde{E}_{i}, \widetilde{E}_{0} \rbrack - \nabla^{S^{2N+1}}_{\widetilde{E_{i}}}\widetilde{E_{0}} {\Huge )}.$
\end{center} 

\medskip 

$\widetilde{h} {\Huge (} \widetilde{U}, \lbrack \widetilde{E}_{i}, \widetilde{E}_{0} \rbrack {\Huge )}$ is zero because $\widetilde{U}$ is normal and $\lbrack \widetilde{E}_{i}, \widetilde{E}_{0} \rbrack$ is tangent to $\widetilde{M}$, so this leaves us with $\widetilde{h} {\Huge (} \widetilde{U}, \nabla^{S^{2N+1}}_{\widetilde{E_{i}}}\widetilde{E_{0}} {\Huge )}$.  This is equal to $\widetilde{h} {\Huge (} A_{\widetilde{u}}(\widetilde{e}_{i}), \widetilde{E}_{0} {\Huge )}$.  Because $\widetilde{e}_{i}$ is a principal vector for $A_{\widetilde{u}}$, as we showed above, this is equal to $\widetilde{h} {\Huge (} \kappa_{i} \widetilde{e}_{i}, \widetilde{E}_{0} {\Huge )} = \kappa_{i} \widetilde{h}(\widetilde{e}_{i}, \widetilde{E}_{0}) = 0$.  This establishes that $\widetilde{h} {\Huge (} \nabla_{\widetilde{E}_{0}}^{S^{2N+1}} \widetilde{U}, \widetilde{E}_{i} {\Huge )} = 0$ and completes the proof that the Hopf fibres are principal directions for $\widetilde{M}$ in $S^{2N+1}$, with principal curvature zero.  

\end{proof}  

The proof of Proposition \ref{tcsame} also uses the following well-known fact, which gives a simpler expression for $\mathcal{T}_{\C P^{N}}(M)$ and is important in some of our later results: 

\begin{proposition}
\label{eqop}

Let $M$ be a complex submanifold of a K\"ahler manifold $P$.  Suppose $\vec{u}$ is a unit normal vector to $M$ at $p$ and $e$ is an principal vector for the second fundamental form $A_{\vec{u}}$, with principal curvature $\kappa$.  Let $J$ denote the complex structure.  Then $J(e)$ is also a principal vector for $A_{\vec{u}}$, with principal curvature $-\kappa$.

\end{proposition}

Proposition \ref{eqop} follows from the fact that the second fundamental form of a K\"ahler submanifold is anti-invariant under the action of the complex structure: $\langle A_{\vec{u}}(J(x)), J(y) \rangle = -\langle A_{\vec{u}}(x), y \rangle$.  A proof of this result together with several applications to the study of K\"ahler submanifolds can be found in \cite{Si}.  Note that Proposition \ref{eqop} implies K\"ahler submanifolds are minimal - in fact, closed K\"ahler submanifolds are minimal in the very strong sense that they have the minimal volume of any submanifold in their homology class.  This is known as \textit{Wirtinger's inequality} (not to be confused with the more famous inequality for periodic functions) and a proof can be found in \cite{F}.   \\ 

In light of Proposition \ref{eqop}, for a complex projective manifold $M$ of complex dimension $m$, we can write its principal curvatures for a normal vector $\vec{u}$ as $\kappa_{1}, -\kappa_{1}, \kappa_{2}, -\kappa_{2}, \cdots , \kappa_{m}, -\kappa_{m}$. We then have:  

\begin{center}
\begin{equation}
\displaystyle det \text{\Large $($}\cos(r)Id_{T_{p}M} - \sin(r) A_{\vec{u}} \text{\Large $)$} = \prod\limits_{i=1}^{m} \text{\large (}\cos^{2}(r) - \kappa_{i}^{2} \sin^{2}(r) \text{\large )}. 
\end{equation}
\end{center}
\medskip

And we have the following two formulas for $\mathcal{T}_{\C P^{N}}(M)$:  

\begin{center}
\begin{equation}
\text{\scriptsize $\frac{2}{Vol(\C P^{N})}$} \displaystyle \text{\LARGE $\int\limits_{\text{\tiny $\nu^{1}M$}} \int\limits_{\text{\small $0$}}^{\text{\small $\frac{\pi}{2}$}}$} \text{\LARGE $\vert$} \prod\limits_{i=1}^{m} \text{\large $($}\cos^{2}(r) - \kappa_{i}^{2} \sin^{2}(r) \text{\large $)$} \text{\LARGE $\vert$} \cos(r) \sin^{(2N-2m-1)}(r)  \text{\em dr $dVol_{\text{\tiny $\nu^{1}M$}}(\vec{u})$},
\end{equation}
\end{center}

\begin{center}
\begin{equation}
\label{formulaforcptc}
\text{\scriptsize $\frac{2}{Vol(\C P^{N})}$} \displaystyle \text{\LARGE $\int\limits_{\text{\tiny $\nu^{1}M$}} \int\limits_{\text{\small $0$}}^{\text{\small $\frac{\pi}{2}$}}$} \text{\LARGE $\vert$} \sum\limits_{i=0}^{m}(-1)^{i}\sin^{(2N-2m-1+2i)}(r)\cos^{(2m-2i+1)}(r)\sigma_{i}(\kappa^{2})  \text{\LARGE $\vert$} \text{\em dr $dVol_{\nu^{1}M}(\vec{u})$}.
\end{equation}
\end{center}
\medskip

Here, as in (\ref{atcsphereformula}), $\sigma_{i}( \kappa^{2} )$ represents the $i^{th}$ elementary symmetric function of the squares of the principal curvatures of the normal vector $\vec{u}$.  In this case, $\sigma_{m}( \kappa^{2} ) = \kappa_{1}^{2}\kappa_{2}^{2}\dots\kappa_{m}^{2}$ is $(-1)^{m}$ times the Gauss curvature.  

\begin{proof}[Proof of Proposition \ref{tcsame}]

Let $\widetilde{u}$ be a unit normal vector to $\widetilde{M}$ in $S^{2N+1}$, with $\vec{u}$ its image via the differential of the Hopf fibration.  $\vec{u}$ is normal to $M$ in $\C P^{N}$.  We let $\widetilde{\kappa}_{0}, \widetilde{\kappa}_{1}, \widetilde{\kappa}_{2}, \cdots \widetilde{\kappa}_{n}$ be the principal curvatures of $\widetilde{u}$, with $\widetilde{\kappa}_{0} = 0$ corresponding to the principal direction along the Hopf fibre.  We let $\kappa_{1}, -\kappa_{1}, \kappa_{2}, -\kappa_{2}, \cdots, \kappa_{m}, -\kappa_{m}$ denote the principal curvatures of $\vec{u}$, so that we have: \\

\begin{center}
$\widetilde{\kappa}_{0} = 0$, $\widetilde{\kappa}_{1} = \kappa_{1}$, $ \widetilde{\kappa}_{2} = -\kappa_{1}$, $ \widetilde{\kappa}_{3} = \kappa_{2}$, $\widetilde{\kappa}_{4} = -\kappa_{2}, \cdots, \widetilde{\kappa}_{n-1} = \kappa_{m}$, $\widetilde{\kappa}_{n} = -\kappa_{m}.$
\end{center}

\medskip

Let $\widetilde{p}$ denote the basepoint of $\widetilde{u}$ in $\widetilde{M}$ and $p$ its image in $M$.  For $r \in [0, \pi]$, we have: \\  

\begin{center}

$ \text{\em \large det} \text{\Large $($} \cos(r)Id_{T_{\widetilde{p}}\widetilde{M}} - \sin(r) A_{\widetilde{u}} \text{\Large $)$} = \displaystyle \prod\limits_{j=0}^{n} {\large (}\cos(r) - \widetilde{\kappa}_{j} \sin(r){\large )}$

\smallskip  

\begin{equation}
\label{detformula}
 = \cos(r) \prod\limits_{i=1}^{m} {\large (}\cos^{2}(r) - \kappa^{2}_{i} \sin^{2}(r){\large )} =  \cos(r) \text{\em \large det} \text{\Large $($} \cos(r)Id_{T_{p}M} - \sin(r) A_{\vec{u}} \text{\Large $)$}.
\end{equation}

\end{center}

\medskip

Because $\cos^{2}(r) - \kappa^{2}_{i} \sin^{2}(r) = \cos^{2}(\pi - r) - \kappa^{2}_{i} \sin^{2}(\pi - r)$ and $\vert \cos(r) \vert = \vert \cos(\pi - r) \vert$, we have: \\ 

\begin{center}
$\displaystyle \text{\LARGE $\int\limits_{\text{\small $0$}}^{\text{\small $\pi$}}$} \text{\LARGE $\vert$} \text{\em \large det} \text{\Large $($} \cos(r)Id_{T_{\widetilde{p}}\widetilde{M}} - \sin(r) A_{\widetilde{u}} \text{\Large $)$} \text{\LARGE $\vert$} \sin^{(2N - n -1)}(r) dr$

\medskip 

$\displaystyle = \text{\large $2$} \text{\LARGE $\int\limits_{\text{\small $0$}}^{\text{\small $\frac{\pi}{2}$}}$} \text{\LARGE $\vert$} \text{\em \large det} \text{\Large $($} \cos(r)Id_{T_{\widetilde{p}}\widetilde{M}} - \sin(r) A_{\widetilde{u}} \text{\Large $)$} \text{\LARGE $\vert$} \sin^{(2N - n -1)}(r) dr. $

\end{center}

\bigskip

Equation (\ref{detformula}) then implies this is equal to the corresponding integral for $\vec{u}$: \\  

\begin{center} 

$\displaystyle 2 \text{\LARGE $\int\limits_{\text{\small $0$}}^{\text{\small $\frac{\pi}{2}$}}$} \text{\LARGE $\vert$} \text{\em \large det} \text{\Large $($} \cos(r)Id_{T_{\widetilde{p}}\widetilde{M}} - \sin(r) A_{\widetilde{u}} \text{\Large $)$} \text{\LARGE $\vert$} \sin^{(2N - n -1)}(r) dr$

\medskip 

$\displaystyle = 2 \text{\LARGE $\int\limits_{\text{\small $0$}}^{\text{\small $\frac{\pi}{2}$}}$} \text{\LARGE $\vert$}  \text{\em \large det} \text{\Large $($} \cos(r)Id_{T_{p}M} - \sin(r) A_{\vec{u}} \text{\Large $)$} \text{\LARGE $\vert$} \cos(r) \sin^{(2N - n -1)}(r) dr.   $

\end{center} 

\medskip

(The additional term $\cos(r)$ in the integrand for $\mathcal{T}_{\C P^{N}}(M)$, which appears in the second term above, is because of the Jacobi field with initial value $J(\vec{u})$ along the geodesic $\gamma_{\vec{u}}$ in $\C P^{N}$.  This is explained in the proof of Proposition \ref{cppc}, in Section \ref{tfgbc}.) \\ 

The Hopf fibres are geodesics of $S^{2N+1}$ of length $2\pi$.  For $p \in M$, we therefore have: \\  

\begin{center}

$\displaystyle 2\pi \displaystyle \text{\LARGE $\int\limits_{\text{\tiny $\nu^{1}_{p}M$}} \int\limits_{\text{\small $0$}}^{\text{\small $\frac{\pi}{2}$}}$} \text{\LARGE $\vert$}  \text{\em \large det} \text{\Large $($} \cos(r)Id_{T_{p}M} - \sin(r) A_{\vec{u}} \text{\Large $)$} \text{\LARGE $\vert$} \cos(r) \sin^{(2N - n -1)}(r) \text{\em dr d$\vec{u}$ }$ 

\medskip  

$= \displaystyle \displaystyle \text{\LARGE $ \int\limits_{\text{\tiny $\pi^{-1}(p)$}} \int\limits_{\text{\tiny $\nu^{1}_{\widetilde{p}}\widetilde{M}$}} \int\limits_{\text{\small $0$}}^{\text{\small $\frac{\pi}{2}$}}$} \text{\LARGE $\vert$} \text{\em \large det} \text{\Large $($} \cos(r)Id_{T_{\widetilde{p}}\widetilde{M}} - \sin(r) A_{\widetilde{u}} \text{\Large $)$} \text{\LARGE $\vert$} \sin^{(2N - n -1)}(r)  \text{\em dr d$\widetilde{u}$ d$\widetilde{p}$}.$

\end{center}

\medskip

Together with the fact that $Vol(S^{2N+1}) = 2\pi Vol(\C P^{N})$, this implies that: \\  

$ \mathcal{T}_{\C P^{N}}(M) = \frac{2}{Vol(\C P^{N})} \displaystyle \text{\LARGE $\int\limits_{\text{\tiny $\nu^{1}M$}} \int\limits_{\text{\small $0$}}^{\text{\small $\frac{\pi}{2}$}}$}  \text{\LARGE $\vert$}  \text{\em \large det} \text{\Large $($} \cos(r)Id_{T_{p}M} - \sin(r) A_{\vec{u}} \text{\Large $)$} \text{\LARGE $\vert$} \cos(r) \sin^{(2N - n -1)}(r)  \text{\em dr $ dVol_{\nu^{1}M}(\vec{u})$}. $

\medskip

$=\frac{1}{Vol(S^{2N+1})} \displaystyle \text{\LARGE $\int\limits_{\text{\tiny $\nu^{1}\widetilde{M}$}} \int\limits_{\text{\small $0$}}^{\text{\small $\pi$}}$}  \text{\LARGE $\vert$} \text{\em \large det} \text{\Large $($} \cos(r)Id_{T_{\widetilde{p}}\widetilde{M}} - \sin(r) A_{\widetilde{u}} \text{\Large $)$} \text{\LARGE $\vert$} \sin^{(2N - n -1)}(r) \text{\em dr $dVol_{\nu^{1}\widetilde{M}}(\widetilde{u})$} = \mathcal{T}_{S^{2N+1}}(\widetilde{M}).$

\end{proof}

In light of Proposition \ref{tcsame}, we can establish an inequality between the total absolute curvature of a complex projective manifold and its Betti numbers with the following result: 

\begin{proposition}
\label{bettinumberrelationship}

Let $M$ be a compact complex manifold, holomorphically immersed in the complex projective space $\C P^{N}$.  Let $\widetilde{M}$ be the $S^{1}$-bundle over $M$ which is induced from the Hopf fibration $\pi: S^{2N+1} \rightarrow \C P^{N}$ by the immersion of $M$ into $\C P^{N}$. \\ 

For $0 \leq k \leq m = dim_{\C}(M)$, we have: \\ 

$\beta_{k}(M;\R) = \beta_{k}(\widetilde{M};\R) + \beta_{k-2}(\widetilde{M};\R) + \beta_{k-4}(\widetilde{M};\R) + \cdots \text{       } = \sum\limits_{i=0}^{\lfloor \frac{k}{2} \rfloor} \beta_{k-2i}(\widetilde{M};\R).$ \\ 

For $m \leq k \leq n = dim_{\R}(M)$, we have: \\ 

$\beta_{k}(M;\R) = \beta_{k+1}(\widetilde{M};\R) + \beta_{k+3}(\widetilde{M};\R) + \beta_{k+5}(\widetilde{M};\R) + \cdots  \text{       } = \sum\limits_{i=0}^{\lfloor \frac{n-k}{2} \rfloor} \beta_{k+2i+1}(\widetilde{M};\R).$

\end{proposition}

\begin{proof}

Let $\widetilde{\pi}:\widetilde{M} \rightarrow M$ denote the $S^{1}$-bundle as above.  The fibres of $\widetilde{M}$ are oriented by the $1$-form $d\theta$ dual to the Hopf vector field, and $\widetilde{M}$ is oriented by $d\theta \wedge \widetilde{\pi}^{*}(dVol_{M})$.  The cohomology of $\widetilde{M}$ is therefore related to the cohomology of $M$ by a Gysin sequence, as follows: 

\bigskip  

\begin{center}
{ $\cdots \rightarrow H^{i}(M;\R) \xrightarrow[]{\bullet \smile \chi} H^{i+2}(M;\R) \xrightarrow[]{\widetilde{\pi}^{*}} H^{i+2}(\widetilde{M};\R) \xrightarrow[]{\sigma^{*}} H^{i+1}(M;\R) \xrightarrow[]{\bullet \smile \chi} H^{i+3}(M;\R) \rightarrow \cdots $  }
\end{center}

\bigskip   

In the homomorphism $H^{i}(M;\R) \xrightarrow[]{\bullet \smile \chi} H^{i+2}(M;\R)$, one takes the cup product with $\chi$, the Euler class of the $2$-disk bundle associated to the $S^{1}$-fibration $\widetilde{\pi}:\widetilde{M} \rightarrow M$.  This bundle is topologically equivalent to the complex line bundle associated to $\widetilde{\pi}:\widetilde{M} \rightarrow M$.  We will denote this bundle by $\widetilde{\pi}: \widetilde{L} \rightarrow M$. \\ 

$\widetilde{L}$ is also induced by the immersion of $M$ into $\C P^{N}$, from the complex line bundle associated to the Hopf fibration, which is $\mathcal{O}(-1) \in Pic(\C P^{N})$.  The Euler class of $\mathcal{O}(-1)$ is its first Chern class, which is $[\frac{-1}{\pi} \omega_{FS}]$, where $\omega_{FS}$ is the K\"ahler form of the metric on $\C P^{N}$.  Because the Euler class of $\widetilde{L}$ is the pull-back of that of $\mathcal{O}(-1)$ and the metric on $M$ is the pull-back of that on $\C P^{N}$, we have $\chi = [\frac{-1}{\pi} \omega]$, where $\omega$ is the K\"ahler form of $M$. \\ 

We let $\mathscr{L}: H^{i}(M;\R) \rightarrow H^{i+2}(M;\R)$ denote the Lefschetz operator, i.e. $[\alpha] \mapsto [\omega] \smile [\alpha] = [\omega \wedge \alpha]$.  The hard Lefschetz theorem states that for $0 \leq k \leq m - 1$, $\mathscr{L}^{m-k} : H^{k}(M;\R) \rightarrow H^{2m-k}(M;\R)$ is an isomorphism.  In particular, $\mathscr{L}: H^{i}(M;\R) \rightarrow H^{i+2}(M;\R)$ is injective for $0 \leq i \leq m-1$ and surjective for $m-1 \leq i \leq n-2$.  It is trivially surjective (zero) for $i = n-1$ and $i = n$. \\

Up to the factor $-\frac{1}{\pi}$, the homomorphism $H^{i}(M;\R) \xrightarrow[]{\bullet \smile \chi} H^{i+2}(M;\R)$ in the Gysin sequence is the Lefschetz operator.  By the hard Lefschetz theorem, it is therefore injective for $0 \leq i \leq m-1$.  This injectivity implies that $\sigma^{*}: H^{i+1}(\widetilde{M};\R) \rightarrow H^{i}(M;\R)$ has trivial image, and this implies that $\widetilde{\pi}^{*}: H^{i+1}(M;\R) \rightarrow H^{i+1}(\widetilde{M};\R)$ is surjective.  For $0 \leq i \leq m-1$, using again the injectivity of the Lefschetz operator in the homomorphism $H^{i-1}(M;\R) \xrightarrow[]{\bullet \smile \chi} H^{i+1}(M;\R)$, we then have the following short exact sequence:  

\begin{equation}
\label{lowerhalfses}
0 \rightarrow H^{i-1}(M;\R) \xrightarrow[]{\bullet \smile \chi} H^{i+1}(M;\R) \xrightarrow[]{\widetilde{\pi}^{*}} H^{i+1}(\widetilde{M};\R) \rightarrow 0.
\end{equation} 

\bigskip     

This implies that $\beta_{i+1}(M;\R) = \beta_{i-1}(M;\R) + \beta_{i+1}(\widetilde{M};\R)$.  The same is true of $\beta_{i-1}(M;\R)$, i.e. it is equal to $\beta_{i-3}(M;\R) + \beta_{i-1}(\widetilde{M};\R)$.  For $0 \leq k \leq m$, we therefore have:     

\begin{equation}  
\label{lowerhalfbettiidentity}
\beta_{k}(M;\R) = \beta_{k}(\widetilde{M};\R) + \beta_{k-2}(\widetilde{M};\R) + \cdots + \beta_{k - 2l}(\widetilde{M};\R) + \cdots 
\end{equation}

\bigskip  

Because $\mathscr{L}: H^{i}(M;\R) \rightarrow H^{i+2}(M;\R)$ is surjective for $m-1 \leq i \leq n-2$, we likewise have $\widetilde{\pi}^{*}:H^{i+2}(M;\R) \rightarrow H^{i+2}(\widetilde{M};\R)$ is zero.  This implies that $\sigma^{*}:H^{i+2}(\widetilde{M};\R) \rightarrow H^{i+1}(M;\R)$ is injective, and we have the following short exact sequence:  

\begin{equation}
\label{upperhalfses}
0 \rightarrow H^{i+2}(\widetilde{M};\R) \xrightarrow[]{\sigma^{*}} H^{i+1}(M;\R) \xrightarrow[]{\bullet \smile \chi} H^{i+3}(M;\R) \rightarrow 0.
\end{equation} 

\bigskip     

For $m \leq k \leq n$ this gives us $\beta_{k}(M;\R) = \beta_{k+1}(\widetilde{M};\R) + \beta_{k+2}(M;\R)$, and as in (\ref{lowerhalfbettiidentity}), for $m \leq k \leq n$ we then have:  

\begin{equation}
\label{upperhalfbettiidentity}
\beta_{k}(M;\R) = \beta_{k+1}(\widetilde{M};\R) + \beta_{k+3}(\widetilde{M};\R) + \cdots + \beta_{k+2l+1}(\widetilde{M};\R) + \cdots 
\end{equation}

\end{proof} 

\begin{remark}
\label{prim_cohom}

The proof of Proposition \ref{bettinumberrelationship} establishes a relationship between the cohomology of $\widetilde{M}$ and the primitive cohomology of $M$:  for $0 \leq k \leq dim_{\C}(M)$, $H^{k}(\widetilde{M};\R)$ is isomorphic to the primitive cohomology of $M$ in dimension $k$.  The short exact sequence in (\ref{lowerhalfses}) gives an isomorphism between $H^{k}(\widetilde{M};\R)$ and the primitive part of $H^{k}(M;\R)$, and the decomposition of $\beta_{k}(M;\R)$ in (\ref{lowerhalfbettiidentity}) corresponds to the Lefschetz decomposition of $H^{k}(M;\R)$.  Many of the results below can therefore be stated in terms of the primitive cohomology of $M$.  

\end{remark}

The results above lead to the following proposition, which is the basic inequality between the total curvature of a complex projective manifold and its Betti numbers.  Part A of Theorem \ref{cpcl}, and several results later in the paper, are based on this proposition:  

\begin{proposition}
\label{basicestimate}

Let $M$ be a compact complex manifold, of complex dimension $m$ (real dimension $n = 2m$) holomorphically immersed in the complex projective space $\C P^{N}$, and let $\beta_{i}(M;\R)$ be the Betti numbers of $M$ with real coefficients.\\  

Then $\beta_{m-1}(M;\R) + 2\beta_{m}(M;\R) + \beta_{m+1}(M;\R) \leq \mathcal{T}(M)$.    

\end{proposition}

\begin{proof} 

Let $\widetilde{\pi} : \widetilde{M} \rightarrow M$ be the $S^{1}$-bundle over $M$ induced the the Hopf fibration as above.  By Theorem \ref{spherecl}.A and Propositions \ref{tcsame} and \ref{bettinumberrelationship}, we have:  

\begin{center}
\begin{equation}
\label{mainequation}
{\small \beta_{m-1}(M;\R) + 2\beta_{m}(M;\R) + \beta_{m+1}(M;\R) = \sum\limits_{j=0}^{n+1}\beta_{j}(\widetilde{M};\R) \leq \mathcal{T}(\widetilde{M}) = \mathcal{T}(M)}.
\end{equation} 
\end{center}
%\medskip  
\end{proof}

\begin{remark}  

The calculations for (\ref{mainequation}) are slightly different, depending on whether $M$ has even or odd complex dimension:  \\ 

If $m$ is even, then $\beta_{m}(M;\R) = \sum\limits_{j=0}^{\text{\tiny $\frac{m}{2}$}}\beta_{2j}(\widetilde{M};\R) = \sum\limits_{j=0}^{\text{\tiny $\frac{m}{2}$}}\beta_{\text{\tiny $(n+1-2j)$}}(\widetilde{M};\R)$. \\  

Because $\beta_{0}(\widetilde{M};\R)$ and $\beta_{n+1}(\widetilde{M};\R)$ are both equal to $1$, this implies that $\beta_{m}(M;\R) \geq 1$. \\  

If $m$ is odd, then $\beta_{m-1}(M;\R) = \sum\limits_{j=0}^{\text{\tiny $\frac{m-1}{2}$}}\beta_{2j}(\widetilde{M};\R)$ and $\beta_{m+1}(M;\R) = \sum\limits_{j=0}^{\text{\tiny $\frac{m-1}{2}$}}\beta_{\text{\tiny $(n+1-2j)$}}(\widetilde{M};\R)$. \\  

Because $\beta_{0}(\widetilde{M};\R) = 1$, we have $\beta_{m-1}(M;\R) \geq 1$, and because $\beta_{n+1}(\widetilde{M};\R) = 1$ we have $\beta_{m+1}(M;\R) \geq 1$. \\ 

The powers of the K\"ahler form give $M$ a non-trivial cohomology class in each of its even-dimensional cohomology groups and imply that its even-dimensional Betti numbers are non-zero - these observations can be seen as a reflection of this fact. 

\end{remark} 

By the hard Lefschetz theorem (or Poincar\'e duality), $\beta_{m-1}(M;\R) = \beta_{m+1}(M;\R)$, so we can also state the conclusion of Proposition \ref{basicestimate} as:   

\begin{center}
\begin{equation}
\label{other_basic_estimate}
\beta_{m}(M;\R) + \beta_{m \pm 1}(M;\R) \leq {\small \frac{\mathcal{T}(M)}{2}}.
\end{equation}
\end{center}
\medskip  

The middle-dimensional Betti number of $M$, and the  Betti numbers in the dimensions $m \pm 1$, are the largest in dimensions with their respective parities.  This gives us the following family of results:  

\begin{proposition}  
\label{detailedestimate}

Let $M$ be as in Proposition \ref{basicestimate}.  Then the sum of any even and any odd-dimensional real Betti number of $M$ is bounded above by $\frac{\mathcal{T}(M)}{2}$.  If $\beta_{2k}(M;\R) + \beta_{2l+1}(M;\R) = \frac{\mathcal{T}(M)}{2}$, then all even-dimensional real Betti numbers of $M$ in dimensions $2k$ through $n-2k$ are equal to $\beta_{2k}(M;\R)$, and all odd-dimensional real Betti numbers of $M$ in dimesions $2l+1$ through $n-2l-1$ are equal to $\beta_{2l+1}(M;\R)$.  

\medskip

In particular, all of the even-dimensional real Betti numbers of $M$ are bounded above by $\frac{\mathcal{T}(M)}{2}$ and all of the odd-dimensional real Betti numbers of $M$ are bounded above by $\frac{\mathcal{T}(M)}{2} - 1$.  If equality holds for an even-dimensional Betti number $\beta_{2k}(M;\R)$, then all odd-dimensional real Betti numbers of $M$ are equal to $0$, and if equality holds for an odd-dimensional Betti number $\beta_{2l+1}(M;\R)$, then all even-dimensional real Betti numbers of $M$ are equal to $1$.  

\end{proposition}

\begin{proof}  

The observations above immediately imply that the sum of an even and an odd Betti number of $M$ is bounded above by $\frac{\mathcal{T}(M)}{2}$.  The statement that $\beta_{2k}(M;\R) + \beta_{2l+1}(M;\R) = \frac{\mathcal{T}(M)}{2}$ is equivalent to the statement that {\small $\beta_{2k}(M;\R) + \beta_{2l+1}(M;\R) + \beta_{n-2k}(M;\R) + \beta_{n-2l-1}(M;\R)  = \mathcal{T}(M)$}.  Letting $\widetilde{M}$ be the $S^{1}$-bundle over $M$ as above, Proposition \ref{bettinumberrelationship} gives us the following relationship between the total curvature and Betti numbers of $\widetilde{M}$:  

\begin{equation}
\label{equalitycharacterization}
\sum\limits_{j=0}^{k}\beta_{2j}(\widetilde{M};\R) + \sum\limits_{j=0}^{l}\beta_{2j+1}(\widetilde{M};\R) + \sum\limits_{j=0}^{k}\beta_{n+1-2j}(\widetilde{M};\R) + \sum\limits_{j=0}^{l}\beta_{n-2j}(\widetilde{M};\R) = \mathcal{T}_{S^{2N+1}}(\widetilde{M}).  
\end{equation}

\medskip  

Part A of Theorem \ref{spherecl} implies that the Betti numbers of $\widetilde{M}$ other than those in (\ref{equalitycharacterization}) are zero.  Proposition \ref{bettinumberrelationship} then implies that the Betti numbers of $M$ in the ranges described above are constant.  

\end{proof}

As a corollary of these results, we have the following statement, which gives a parallel to the first Chern-Lashof theorem for complex projective manifolds: \\ 

\textbf{Theorem \ref{cpcl}.A}  {\em Let $M$ be a compact complex manifold, of complex dimension $m$, holomorphically immersed in complex projective space.  Let $\mathcal{T}(M)$ be its total absolute curvature and $\beta_{i}$ its Betti numbers with real coefficients.  Then $\sum\limits_{i=0}^{2m}\beta_{i} \leq (\frac{m+1}{2})\mathcal{T}(M)$.  In particular, $\mathcal{T}(M) \geq 2$.} \\ 

\begin{proof}[Proof of Part A of Theorem \ref{cpcl}]

Each of the terms $\beta_{2i-1} + \beta_{2i}$, for $i = 1, 2, \cdots, m$, is bounded above by $\frac{\mathcal{T}(M)}{2}$.  $\beta_{0} = 1$ is likewise, and this implies the result.  

\end{proof}

Theorem \ref{cpcl}.A is sharp, in that equality holds for linearly embedded complex projective subspaces, but these are the only complex projective manifolds for which this equality holds.  This follows from Part B of Theorem \ref{cpcl}.  It would be more interesting to characterize equality in Proposition \ref{basicestimate}.  In their second paper, Chern and Lashof gave the following characterization of equality in their first theorem:  

\begin{theorem}[Chern-Lashof, \cite{CLII}]

Let $M^{n}$ be a compact manifold immersed in Euclidean space with {\small $\sum\limits_{i=0}^{n} \beta_{i}(M;\R) = \mathcal{T}(M)$}. \\ 

Then the integral homology groups of $M$ are torsion-free.  

\end{theorem}

By Proposition \ref{cltcequalsspheretc}, the equivalent theorem holds for submanifolds of spheres.  Equality in Proposition \ref{basicestimate} therefore implies that the integral homology groups of the spherical pre-image $\widetilde{M}$ are torsion-free. \\ 

We note that if a complex projective manifold $M$ is a complete intersection, its Betti numbers can be expressed in terms of the degrees of equations generating its ideal.  Proposition \ref{basicestimate} can then be translated into a statement in terms of these degrees.  We also note that the proofs of the first Chern-Lashof theorem and the corresponding result for submanifolds of spheres in Theorem \ref{spherecl}.A actually imply the following stronger result: letting $c(M)$ denote the minimum number of cells in a cell complex homotopy-equivalent to $M$, $c(M) \leq \mathcal{T}(M)$.  Chern and Lashof discuss this in their second paper \cite{CLII}.  Spheres and complex projective spaces can be immersed with the minimum possible total absolute curvature, equal to $2$, in Theorems \ref{spherecl} and \ref{cpcl} - we will discuss this in Section \ref{minimizers}.  $c(S^{n}) = 2$ for all $n$, however $c(\C P^{m}) = m+1$.  This implies that the stronger conclusion above does not extend directly to complex projective manifolds. 

%%%%%%%%%%%%%%%%%%%%%%%%%%%%%%%%%%%%%%%%%%%%%%%%%%%%%%%%%%%%%%%%%%%%%%%%%%%%%%%%%%%%%%%%%%%%%%%%%%%%%%%%%%%%%%%%%%%%%%%%%

\section{Complex Projective Manifolds with Minimal Total Curvature}
\label{minimizers}

The complex projective manifolds with minimal total absolute curvature are characterized as follows: \\ 

\textbf{Theorem \ref{cpcl}.B} {\em Let $M$ be a compact complex manifold, holomorphically immersed in complex projective space.  If $\mathcal{T}(M) < 4$, then in fact $\mathcal{T}(M) = 2$.  This occurs precisely if $M$ is a linearly embedded complex projective subspace.} \\ 

In Example \ref{fermatconic}, we will see that this result is sharp.  

\begin{proof}[Proof of Part B of Theorem \ref{cpcl}]  

Let $\widetilde{\pi} : \widetilde{M} \rightarrow M$ be the $S^{1}$-bundle over $M$ which is induced by the Hopf fibration, as in Section \ref{bettis}.  By Proposition \ref{tcsame}, $\widetilde{M}$ is immersed in $S^{2N+1}$ with $\mathcal{T}_{S^{2N+1}}(\widetilde{M}) < 4$.  $\widetilde{M}$ has odd dimension $2m+1$, so by Theorem \ref{strongerspherecl2}, $\widetilde{M}$ is homeomorphic to $S^{2m+1}$.  The Gysin sequence for the fibration $\widetilde{\pi}: S^{2m+1} \rightarrow M$ with integral cohomology is therefore as follows:  

\medskip 
\begin{center} 
{ $\cdots \rightarrow H^{i+1}(S^{2m+1};\Z) \xrightarrow[]{\sigma^{*}} H^{i}(M;\Z) \xrightarrow[]{\bullet \smile \chi} H^{i+2}(M;\Z) \xrightarrow[]{\widetilde{\pi}^{*}} H^{i+2}(S^{2m+1};\Z) \rightarrow \cdots $  }
\end{center}
\bigskip

As in the proof of Proposition \ref{bettinumberrelationship}, $\chi$ is the Euler class of the $2$-disk bundle associated to the $S^{1}$-fibration $\widetilde{M} \rightarrow M$.  Because $H^{i}(S^{2m+1};\Z) = 0$ for $i \ne 0, 2m+1$, we have that $\bullet \smile \chi: H^{i}(M;\Z) \rightarrow H^{i+2}(M;\Z)$ is an isomorphism for $i=0$, $1$, $2$, $\cdots$, $2m-2$.  This implies that for $k = 0$, $1$, $2$, $\cdots$, $m$, $H^{2k}(M;\Z)$ is infinite cyclic, generated by $\chi^{k}$. \\ 

In the proof of Proposition \ref{bettinumberrelationship}, we observed that $\chi$ is equal to $[-\frac{1}{\pi} \omega ]$, where $\omega$ is the K\"ahler form of $M$.  This implies that $[-\frac{1}{\pi} \omega ]^{m}$ generates the top-dimensional cohomology $H^{2m}(M;\Z)$ of $M$.  If $M$ is a degree $d$ subvariety of $\C P^{N}$, then $[-\frac{1}{\pi} \omega ]^{m}$ is $d$ times a generator in $H^{2m}(M;\Z)$, because $[\frac{1}{\pi} \omega_{FS}]^{m}$ generates $H^{2m}(\C P^{N};\Z)$.  We have seen that $[-\frac{1}{\pi} \omega ]^{m}$ generates $H^{2m}(M;\Z)$, so $d = 1$ and $M$ is a linearly embedded subspace of $\C P^{N}$. \\  

If $\C P^{m}$  is a linearly embedded subspace of $\C P^{N}$, its pre-image via the Hopf fibration is a totally geodesic $S^{2m+1}$ in $S^{2N+1}$.  By Proposition \ref{tcsame}, $\mathcal{T}_{\C P^{N}}(\C P^{m}) = \mathcal{T}_{S^{2N+1}}(S^{2m+1})$, and by Part C of Theorem \ref{spherecl}, $\mathcal{T}_{S^{2N+1}}(S^{2m+1}) = 2$.  

\end{proof}

\begin{remark} 

The proof of Theorem \ref{cpcl}.B implies that linear subspaces are the only complex projective manifolds $M$ whose pre-images $\widetilde{M}$ via the Hopf fibration are integral homology spheres.  

\end{remark} 

We note that the spherical pre-images $\widetilde{M}$ in the proofs of Parts A and B of Theorem \ref{cpcl} are minimal submanifolds of $S^{2N+1}$.  In light of the results above and in Theorem \ref{spherecl}.C, $\widetilde{M}$ is a totally geodesic submanifold of $S^{2N+1}$ precisely if $M$ is a linear subspace of $\C P^{N}$, in which case $\mathcal{T}(\widetilde{M}) = 2$.  Otherwise, $\widetilde{M}$ has total absolute curvature at least $4$.  In \cite{Si}, Simons proved that all $n$-dimensional closed minimal subvarieties of the round sphere $S^{N}$ have index at least $N-n$, and nullity at least $(n+1)(N-n)$, with equality only for totally geodesic round subspheres.  It would be interesting to know if the apparent similarity between these conclusions - that totally geodesic sub-spheres are isolated within families of minimal submanifolds of the sphere - indicates a stronger connection between these results, or the existence of other results of this type. \\  

We also note that for many $m \geq 2$, it is known that complex projective space is not the only $m$-dimensional compact complex manifold whose real cohomology ring is isomorphic to $\R[\alpha] / [\alpha]^{m+1}$, with $[\alpha]$ a cohomology class in $H^{2}(M;\R)$.  Compact complex manifolds which have the same real Betti numbers as $\C P^{m}$ are known as {\em fake projective spaces}.  In complex dimension $2$, there are known to be fifty fake projective planes, up to homeomorphism, and one hundred up to biholomorphism.  These were classified by work of Prasad and Yeung and Cartwright and Steger in \cite{PY} and \cite{CS}, and all of them can be realized as smooth algebraic surfaces.  Part B of Theorem \ref{cpcl} implies that the total absolute curvature of these spaces, when realized as complex projective manifolds, is at least twice that of complex projective space. 

%%%%%%%%%%%%%%%%%%%%%%%%%%%%%%%%%%%%%%%%%%%%%%%%%%%%%%%%%%%%%%%%%%%%%%%%%%%%%%%%%%%%%%%%%%%%%%%%%%%%%%%%%%%%%%%%%%%%%%%%%

\section{Total Curvature, the Complex Projective Tube Formula and the Gauss-Bonnet-Chern Theorem}
\label{tfgbc}

We begin this section by proving Proposition \ref{cppc}, which shows that total absolute curvature has the same meaning for a complex projective manifold as Chern and Lashof's invariant for a submanifold of Euclidean space: \\     

\begin{proof}[Proof of Proposition \ref{cppc}]   

We begin by verifying that $\mathcal{T}(M) = \displaystyle \text{\scriptsize $\frac{2}{Vol(\C P^{N})}$} \text{\LARGE $\int\limits_{\text{\scriptsize $\nu^{<\frac{\pi}{2}}M$}}$} \text{\LARGE $\vert$} det(dExp^{\perp}) \text{\LARGE $\vert$} dVol_{\text{\tiny $\nu^{<\frac{\pi}{2}}M$}}$. \\ 

As with submanifolds of spheres in Section \ref{sphereresults}, we can calculate the differential of the normal exponential map of a complex projective manifold $M$ using Jacobi fields, and we can express the result entirely in terms of the second fundamental form: let $\vec{u}$ be a unit normal vector to $M$ at $p$, and let $e_{1}, ... , e_{n}$ be an orthonormal basis of principal vectors of the second fundamental form $A_{\vec{u}}$ with principal curvatures $\kappa_{1}, ... , \kappa_{n}$.  Let $u_{2} = J(\vec{u}), u_{3}, ... , u_{\text{\small $2N-n$}}$ be an orthonormal basis for the orthogonal subspace to $\vec{u}$ in $\nu_{p}M$, with $J$ the complex structure of $\C P^{N}$.  Letting $\gamma_{\vec{u}}$  be the geodesic of $\C P^{N}$ through $\vec{u}$, and $E_{1}, ... , E_{n}$ the parallel vector fields along $\gamma_{\vec{u}}$ with initial conditions $e_{1}, ... , e_{n}$, and $F_{2}, ... , F_{2N-n}$ the parallel vector fields along $\gamma_{\vec{u}}$ with initial conditions $J(u) = u_{2}, u_{3}, ... , u_{2N-n}$, we have: \\ 

\begin{itemize}

\item $(dExp^{\perp})_{r\vec{u}}(e_{i}) = {\huge (} \cos(r) - \kappa_{i}\sin(r) {\huge )} E_{i}(r)$ for $i = 1, ..., n$, corresponding to the Jacobi field $J_{i}(t)$ along $\gamma_{\vec{u}}$ with $J_{i}(0) = e_{i}$ and $J_{i}'(0) = -\kappa_{i} e_{i}$.  

\bigskip 

\item $(dExp^{\perp})_{r\vec{u}}(u_{2}) = (\frac{\sin(2r)}{2r})F_{2}(r) = (\frac{\sin(r)\cos(r)}{r})F_{2}(r)$, corresponding to the Jacobi field $J_{n+2}(t)$ along $\gamma_{\vec{u}}$ with $J_{n+2}(0) = 0$ and $J_{n+2}'(0) = \frac{1}{r}u_{2}$.  Note that because $\vec{u}$ and $u_{2} = J(\vec{u})$ span a holomorphic section with curvature $4$, $J_{n+2}(t)$ behaves like a Jacobi field in a space with constant curvature $4$.  

\bigskip 

\item $(dExp^{\perp})_{r\vec{u}}(u_{j}) = (\frac{\sin r}{r})F_{j}(r)$ for $j = 3, ... , 2N-n$, corresponding to the Jacobi field $J_{n+j}(t)$ along $\gamma_{\vec{u}}$ with $J_{n+j}(0) = 0$ and $J_{n+j}'(0) = u_{j}$.  

\bigskip 

\item $(dExp^{\perp})_{r\vec{u}}(\vec{u}) = \gamma_{\vec{u}}'(r)$, corresponding to the Jacobi field $\gamma_{\vec{u}}'(r)$.  

\end{itemize}

\medskip 

We therefore have:  

\begin{center}

$det(dExp^{\perp})_{r\vec{u}} = \prod\limits_{i=1}^{n}(\cos(r) - \kappa_{i}\sin(r))(\frac{\cos(r)\sin^{(2N-n-1)}(r)}{r^{2N-n-1}})$ 
\medskip 

$ = \text{\em \large det} \left( \cos(r)Id_{T_{p}M} - \sin(r) A_{\vec{u}} \right) (\frac{\cos(r)\sin^{(2N-n-1)}(r)}{r^{2N-n-1}}).$ 

\end{center}

\medskip  

For $r$ less than $\frac{\pi}{2}$, this implies that: \\  

\begin{center} 

$| det(dExp^{\perp}) | = | det \left( \cos(r)Id_{T_{p}M} - \sin(r) A_{\vec{u}} \right) |  (\frac{\sin^{(2N-n-1)}(r)\cos(r)}{r^{2N-n-1}}).$ \\ 

\end{center} 

This implies the first equation in Proposition \ref{cppc}, that $\mathcal{T}(M) =  \displaystyle \text{\footnotesize $\frac{2}{Vol(\C P^{N})}$} \text{\LARGE $\int\limits_{\text{\scriptsize $\nu^{<\frac{\pi}{2}}M$}}$} \text{\LARGE $\vert$} det(dExp^{\perp}) \text{\LARGE $\vert$} dVol_{\text{\tiny $\nu^{<\frac{\pi}{2}}M$}}$. \\  

The verification of the second equation is similar to the verification  of the corresponding fact for submanifolds of spheres, in the proof of Parts A and B of Theorem \ref{spherecl}: we let $\widetilde{\nu}M \rightarrow M$ be the bundle over $M$ whose fibre at $p$ is the linear subspace $\C P^{N-m}$ of $\C P^{N}$ orthogonal to $M$ at $p$.  As in the proof of Theorem \ref{spherecl}.A and B, $Exp^{\perp}: \nu^{<\frac{\pi}{2}}M \rightarrow \C P^{N}$ extends to a map $\widetilde{Exp}^{\perp}: \widetilde{\nu}M \rightarrow \C P^{N}$, and because the regular values of $Exp^{\perp}$ contain those of $\widetilde{Exp}^{\perp}$, in addition to being of full measure, they contain an open, dense subset of $\C P^{N}$.  We let $\C P^{N}_{reg}$ denote the regular values of $Exp^{\perp}: \nu^{<\frac{\pi}{2}}M \rightarrow \C P^{N}$. \\  

As in the proof of Theorem \ref{spherecl}.A and B, if $\vec{v}$ is a regular point of $Exp^{\perp}$, there are neighborhoods $V$ of $\vec{v}$ and $Q$ of $Exp^{\perp}(\vec{v})$ such that $Exp^{\perp}: V \rightarrow Q$ is a diffeomorphism.  As in the proof of Theorem \ref{spherecl}.A and B, this allows us to define a positive measure $d\mu$ on $\nu^{<\frac{\pi}{2}}M$ which is absolutely continuous with respect to $dVol_{\nu^{<\frac{\pi}{2}}M}$ and is the pull-back of the measure on $\C P^{N}$ via $Exp^{\perp}$, by integrating a measurable function $\phi$ against $|det(dExp^{\perp})|dVol_{\nu^{<\frac{\pi}{2}}M}$.  For neighborhoods $V$ and $Q$ as above, we have: \\  

\begin{center}
$\displaystyle \text{\LARGE $\int\limits_{\text{\scriptsize $V$}}$} \phi d\mu = \text{\LARGE $\int\limits_{\text{\scriptsize $V$}}$} \phi |det(dExp^{\perp})|dVol_{\nu^{<\frac{\pi}{2}}M} = \text{\LARGE $\int\limits_{\text{\scriptsize $Q$}}$} \phi \circ (Exp^{\perp})^{-1} dVol_{\C P^{N}}.$ 
\end{center}

\medskip

This implies that $(Exp^{\perp})^{-1}(\C P^{N}_{reg})$ is of full measure with respect to $d\mu$, and this implies that: \\  

\begin{center} 

$\displaystyle \text{\scriptsize $\frac{2}{Vol(\C P^{N})}$} \text{\LARGE $\int\limits_{\text{\scriptsize $\nu^{<\frac{\pi}{2}}M$}}$} \text{\LARGE $\vert$} det(dExp^{\perp}) \text{\LARGE $\vert$} dVol_{\text{\tiny $\nu^{<\frac{\pi}{2}}M$}} = \text{\scriptsize $\frac{2}{Vol(\C P^{N})}$} \text{\LARGE $\int\limits_{\text{\scriptsize $(Exp^{\perp})^{-1}(\C P^{N}_{reg})$}}$} \text{\LARGE $\vert$} det(dExp^{\perp}) \text{\LARGE $\vert$} dVol_{\text{\tiny $\nu^{<\frac{\pi}{2}}M$}}$

\bigskip  

$\displaystyle = \text{\scriptsize $\frac{2}{Vol(\C P^{N})}$} \text{\LARGE $\int\limits_{\text{\scriptsize $\C P^{N}_{reg}$}}$} \sharp(Exp^{\perp})^{-1}(q) dVol_{\text{\tiny $\C P^{N}$}}$ 

\end{center}

\medskip  

Because $\C P^{N}_{reg}$ is of full measure in $\C P^{N}$, this implies the result: \\  

\begin{center}

$\displaystyle \text{\scriptsize $\frac{2}{Vol(\C P^{N})}$} \text{\LARGE $\int\limits_{\text{\scriptsize $\nu^{<\frac{\pi}{2}}M$}}$} \text{\LARGE $\vert$} det(dExp^{\perp}) \text{\LARGE $\vert$} dVol_{\text{\tiny $\nu^{<\frac{\pi}{2}}M$}} = \text{\scriptsize $\frac{2}{Vol(\C P^{N})}$} \text{\LARGE $\int\limits_{\text{\scriptsize $\C P^{N}$}}$} \sharp(Exp^{\perp})^{-1}(q) dVol_{\text{\tiny $\C P^{N}$}}$

\end{center} 

\end{proof}

\begin{remark}
\label{mindifference}

Proposition \ref{cppc} illustrates one of the differences between the characterization of minimal total absolute curvature for complex projective manifolds, in Part B of Theorem \ref{cpcl}, and the corresponding result for submanifolds of spheres in Part C of Theorem \ref{spherecl}: if $S^{n}$ is a totally geodesic subsphere of $S^{N}$, then almost all points $\widetilde{q}$ in the ambient space $S^{N}$ have precisely two pre-images via $Exp^{\perp} : \nu^{<\pi}S^{n} \rightarrow S^{N}$.  On the other hand, if $\C P^{m}$ is a linear subspace of $\C P^{N}$, then almost all $q$ in the ambient $\C P^{N}$ have precisely one pre-image via $Exp^{\perp} : \nu^{<\frac{\pi}{2}} \C P^{m} \rightarrow \C P^{N}$.  Thus, $Exp^{\perp} : \nu^{<\pi}S^{n} \rightarrow S^{N}$ covers $S^{N}$  twice and $Exp^{\perp} : \nu^{<\frac{\pi}{2}} \C P^{m} \rightarrow \C P^{N}$ covers $\C P^{N}$ once. 

\end{remark}

By Proposition \ref{cppc}, the total absolute curvature of a complex projective manifold $M$ results from integrating $|det(dExp^{\perp})|$ over $\nu^{<\frac{\pi}{2}}M$.  For small enough normal vectors - for normal vectors $r \vec{u}$ with $||\vec{u}|| = 1$ and $r$ less than the focal radius of $M$ along $\vec{u}$ - we have $det(dExp^{\perp})_{r \vec{u}}$ positive provided that $\C P^{N}$ and $\nu M$ are oriented consistently, for example by their complex structures.  For a compact manifold $M$ holomorphically embedded in $\C P^{N}$, there is then a positive $r_{0}$ such that, letting $\nu^{< r_{0}}M$ be the normal disk bundle to $M$ of radius $r_{0}$, $Exp^{\perp} : \nu^{< r_{0}}M \rightarrow \C P^{N}$ is an orientation-preserving diffeomorphism onto its image.  In this case, we have: \\  

\begin{center}

$\displaystyle \text{\LARGE $\int\limits_{\text{\scriptsize $\nu^{< r_{0}}M$}}$} \text{\LARGE $\vert$} det(dExp^{\perp}) \text{\LARGE $\vert$} dVol_{\text{\tiny $\nu^{< r_{0}}M$}} = \text{\LARGE $\int\limits_{\text{\scriptsize $\nu^{< r_{0}}M$}}$} det(dExp^{\perp}) dVol_{\text{\tiny $\nu^{< r_{0}}M$}} = Vol \left( Exp^{\perp}(\nu^{< r_{0}}M) \right)$; 

\end{center}

\medskip 

in other words, the integral of the total absolute curvature integrand over the normal disk bundle of radius $r_{0}$ gives the volume of the tube of radius $r_{0}$ about $M$ in $\C P^{N}$.  The complex projectiv generalization of Weyl's tube formula, due independently to Wolf (\cite{Wo}) and Flaherty (\cite{Fl}) and reformulated and extended by Gray and Vanhecke (\cite{GV}), states that in this situation, the tube volume depends only on the intrinsic geometry (and codimension) of $M$, and is given in terms of the curvature tensor of $M$ by a formula similar to that in Weyl's formula. \\ 
	
We note that Calabi's theorem on the rigidity of holomorphic embeddings into complex space forms, in \cite{Ca}, already implies that the tube volume in this situation is intrinsic to $M$ - the author is not aware of a reference to this fact in the literature on tube formulas, and we believe it provides interesting context for the tube formula for complex projective manifolds. \\ 

The specific form of the tube formula for a complex projective manifold is also interesting in light of the relationship between Weyl's tube formula, the Gauss-Bonnet-Chern theorem and the Chern-Lashof theorems described in Section \ref{sphereresults}.  Alfred Gray showed that the tube formula for complex projective manifolds can be formulated in terms of the Chern classes of the manifold in question as follows:  

\begin{theorem}[Gray, \cite{GrII}.  See also \cite{Gr}]
\label{gtf}  

Let $M$ be a compact complex manifold of complex dimension $m$ holomorphically embedded in the complex projective space $\C P^{N}$ and suppose that the tube of radius $r_{0}$ about $M$ is the diffeomorphic image of $\nu^{< r_{0}} M$ by $Exp^{\perp}$.  Let $\xi_{1}, \xi_{2}, \dots, \xi_{m}$ be the Chern roots of $M$, in other words, the (possibly non-homogeneous) differential forms which factor the Chern polynomial of $M$ as follows: \\ 

\begin{center}
$\displaystyle 1 + t \gamma_{1} + t^{2} \gamma_{2} + \cdots + t^{m} \gamma_{m} = \prod\limits_{i=1}^{m} \left( 1 + t\xi_{i} \right)$, 
\end{center}

\bigskip

where $\gamma_{1}, \gamma_{2}, \dots, \gamma_{m}$ are the Chern forms of $M$.  Then: \\  

\begin{center}
$\displaystyle Vol \left( Exp^{\perp}(\nu^{< r_{0}}M) \right) = \frac{1}{N!} \text{\LARGE $\int\limits_{\text{\scriptsize $M$}}$} \prod\limits_{i = 1}^{m} \left( 1 - \text{\scriptsize $\frac{1}{\pi}$} \omega + \xi_{i} \right) \wedge \left( \pi \sin^{2}(r_{0}) + \cos^{2}(r_{0}) \omega \right)^{N}$.  
\end{center} 

\smallskip 

In the formula above, $\omega$ is the K\"ahler form of $M$.  We integrate the degree-$2m$ component of the integrand over $M$ and disregard the other terms.  

\end{theorem}

Since $Exp^{\perp} : \nu^{<r_{0}}M \rightarrow \C P^{N}$ is an orientation-preserving diffeomorphism, we can also formulate Theorem \ref{gtf} as: \\

\begin{center}
\begin{equation}
\label{altgtf}
\displaystyle \text{\LARGE $\int\limits_{\text{\scriptsize $\nu^{< r_{0}}M$}}$} Exp^{\perp*} \left( dVol_{\C P^{N}} \right) = \frac{1}{N!} \text{\LARGE $\int\limits_{\text{\scriptsize $M$}}$} \prod\limits_{i = 1}^{m} \left( 1 - \text{\scriptsize $\frac{1}{\pi}$} \omega + \xi_{i} \right) \wedge \left( \pi \sin^{2}(r_{0}) + \cos^{2}(r_{0}) \omega \right)^{N}. 
\end{equation}
\end{center}

\medskip  

Although the integral in (\ref{altgtf}) is no longer equal to a tube volume for $r_{0}$ greater than the focal radius of $M$, Equation (\ref{altgtf}) is valid for all $r_{0} > 0$.  In particular, when $r_{0}$ is equal to $\frac{\pi}{2}$ (the diameter and injectivity radius of $\C P^{N}$) we have: \\ 
%Although the integral in (\ref{altgtf}) is no longer equal to a tube volume for $r_{0}$ greater than the maximum value for which $Exp^{\perp} : \nu^{<r_{0}}M \rightarrow \C P^{N}$ is a diffeomorphism, Equation (\ref{altgtf}) is valid for all $r_{0} > 0$.  In particular, when $r_{0}$ is equal to $\frac{\pi}{2}$, the diameter and injectivity radius of $\C P^{N}$, we have: \\ 

\begin{center}
$\displaystyle \text{\LARGE $\int\limits_{\text{\scriptsize $\nu^{< \frac{\pi}{2}}M$}}$} Exp^{\perp*} \left( dVol_{\C P^{N}} \right) = \frac{\pi^{N}}{N!} \text{\LARGE $\int\limits_{\text{\scriptsize $M$}}$} \prod\limits_{i = 1}^{m} \left( 1 - \text{\scriptsize $\frac{1}{\pi}$} \omega + \xi_{i} \right) $

\bigskip  

\begin{equation}  
\label{cptccf}
\displaystyle = Vol(\C P^{N}) \text{\LARGE $\int\limits_{\text{\scriptsize $M$}}$} \sum\limits_{i = 0}^{m} \left( 1 - \text{\scriptsize $\frac{1}{\pi}$}\omega \right)^{(m - i)} \text{\Large $\wedge \gamma_{i}$}  = Vol(\C P^{N}) \text{\LARGE $\int\limits_{\text{\scriptsize $M$}}$} \sum\limits_{i = 0}^{m} \left(\text{\scriptsize $\frac{-1}{\pi} $}\right)^{(m-i)} \text{\Large $\gamma_{i} \wedge \omega^{\text{\scriptsize $(m-i)$}}$}. 
\end{equation}
\end{center} 

\medskip  

As noted in the proof of Proposition \ref{bettinumberrelationship}, $[\frac{-1}{\pi} \omega]$ is the Chern class of the line bundle pulled back from $\mathcal{O}(-1) \in Pic(\C P^{N})$ by the embedding of $M$ in projective space.  Letting $\widetilde{L}$ denote this bundle and $TM$ the tangent bundle of $M$, $\sum\limits_{i = 0}^{m} \left( \frac{-1}{\pi} \omega \right)^{m - i} \wedge \gamma_{i}$ therefore represents the top-dimensional Chern class of $TM \otimes \widetilde{L}$, and therefore its Euler class.  The integral in (\ref{cptccf}) therefore gives the Euler number of $TM \otimes \widetilde{L}$, which we write $e(TM \otimes \widetilde{L})$.  Noting that $(\ref{altgtf})$ and $(\ref{cptccf})$ hold for immersed as well as embedded $M$, we then have the following parallel to the extrinsic formulation of the Gauss-Bonnet-Chern theorem in Section \ref{sphereresults}:  

\begin{theorem}
\label{cpgbc}
Let $M$ be a compact complex manifold holomorphically immersed in $\C P^{N}$ and $\widetilde{L}$ the line bundle on $M$ pulled back from $\mathcal{O}(-1)$ by the immersion.  Then: \\ 

\begin{center}
$\displaystyle \text{\scriptsize $\frac{1}{Vol(\C P^{N})}$} \text{\LARGE $\int\limits_{\text{\scriptsize $\nu^{< \frac{\pi}{2}}M$}}$} det(dExp^{\perp}) dVol_{\text{\tiny $\nu^{< \frac{\pi}{2}}M$}} = e(TM \otimes \widetilde{L})$. 
\end{center}
%the Euler number of $TM \otimes \widetilde{L}$. 
\end{theorem}

This implies that the degree of the map $\widetilde{Exp}^{\perp}: \widetilde{\nu}M \rightarrow \C P^{N}$ in the proof of Proposition \ref{cppc} is the Euler number of $TM \otimes \widetilde{L}$, just as the Gauss-Bonnet-Chern theorem implies the degree of the Gauss map of a submanifold of Euclidean space is its Euler characteristic.  

%This implies that the degree of the map $\widetilde{Exp}^{\perp}: \widetilde{\nu}M \rightarrow \C P^{N}$ in the proof of Proposition \ref{cppc} is the Euler number of $TM \otimes \widetilde{L}$, just as the Gauss-Bonnet-Chern theorem implies that the degree of the Gauss map on the unit normal bundle of a submanifold of Euclidean space is its Euler characteristic.  

%%%%%%%%%%%%%%%%%%%%%%%%%%%%%%%%%%%%%%%%%%%%%%%%%%%%%%%%%%%%%%%%%%%%%%%%%%%%%%%%%%%%%%%%%%%%%%%%%%%%%%%%%%%%%%%%%%%%%%%%%

\section{Total Curvature and the Geometry of Complex Projective Hypersurfaces}
\label{hypersurfaces}

We can give a more detailed explanation of the relationship between the total absolute curvature of a complex projective hypersurface and other aspects of its geometry.  The key to these results is the following observation:  

\begin{proposition}
\label{IIrelationship}  

Let $M$ be a complex submanifold of a K\"ahler manifold $P$, and let $J$ denote its complex structure.  Let $\vec{u}$ be a normal vector to $M$ at a point $p$, and let $e$ and $e^{*} = J(e)$ be principal vectors for the second fundamental form $A_{\vec{u}}$, with principal curvatures $\kappa$ and $-\kappa$, as in Proposition \ref{eqop}. \\ 

Then {\small $\frac{e + e^{*}}{\sqrt{2}}$} and {\small $\frac{- e + e^{*}}{\sqrt{2}}$} are principal vectors for $J(\vec{u})$, with principal curvatures $\kappa$ and $-\kappa$ respectively. \\ 

In general, letting $\vec{u}_{\theta} = \cos(\theta)\vec{u} + \sin(\theta)J(\vec{u})$, ({\small $\cos(\frac{\theta}{2})e + \sin(\frac{\theta}{2})e^{*}$}) and ({\small $- \sin(\frac{\theta}{2})e + \cos(\frac{\theta}{2})e^{*}$}) are principal vectors for $\vec{u}_{\theta}$ with principal curvatures $\kappa$ and $-\kappa$ respectively.  

\end{proposition}  

\begin{proof}

Let $U$ be a normal vector field on a neighborhood of $p$ which extends $\vec{u}$.  Because $P$ is K\"ahler, so that its connection $\nabla^{P}$ commutes with the complex structure $J$, \\ 

\begin{center}
$A_{J(\vec{u})}(e) = - \left( \nabla_{e}^{P}J(U) \right)^{\top} = -J(\nabla_{e}^{P}U${\large $)^{\top}$} $=$ {\large $-J($} $-A_{\vec{u}}(e)$ {\large $)$} $= J(\kappa e) = \kappa e^{*}.$ 
\end{center}

\medskip  

Similarly, $A_{J(\vec{u})}(e^{*}) = -\kappa (-e) = \kappa e$.  This implies that {\small $\frac{e + e^{*}}{\sqrt{2}}$} and {\small $\frac{- e + e^{*}}{\sqrt{2}}$} are principal vectors for $A_{J(\vec{u})}$: \\ 

\begin{center}

$A_{J(\vec{u})}$({\small $\frac{e + e^{*}}{\sqrt{2}}$}) $= \kappa$({\small $\frac{e + e^{*}}{\sqrt{2}}$}), $A_{J(\vec{u})}$({\small $\frac{- e + e^{*}}{\sqrt{2}}$}) $= -\kappa$({\small $\frac{- e + e^{*}}{\sqrt{2}}$}).

\end{center}

\medskip  

More generally, \\ 

\begin{center}

$A_{\vec{u}_{\theta}}$({\small $\cos(\frac{\theta}{2})e + \sin(\frac{\theta}{2})e^{*}$}) $= \cos(\frac{\theta}{2})\kappa$({\small $\cos(\theta)e + \sin(\theta)e^{*}$})$+ \sin(\frac{\theta}{2})\kappa$({\small $\sin(\theta)e - \cos(\theta)e^{*}$})

\bigskip  

$= \kappa(\cos(\theta)\cos(\frac{\theta}{2}) + \sin(\theta)\sin(\frac{\theta}{2}))e + \kappa(\sin(\theta)\cos(\frac{\theta}{2}) - \sin(\frac{\theta}{2})\cos(\theta))e^{*}$

\bigskip

$= \kappa(\cos(\theta - \frac{\theta}{2})e + \sin(\theta - \frac{\theta}{2})e^{*}) = \kappa(\cos(\frac{\theta}{2})e + \sin(\frac{\theta}{2})e^{*}).$

\end{center}
\medskip 

This implies that {\small $\cos(\frac{\theta}{2})e + \sin(\frac{\theta}{2})e^{*}$} is a principal vector for $A_{\vec{u}_{\theta}}$ with principal curvature $\kappa$.  By a similar calculation, or by Proposition $\ref{eqop}$, {\small $- \sin(\frac{\theta}{2})e + \cos(\frac{\theta}{2})e^{*}$} is a principal vector with principal curvature $-\kappa$.  

\end{proof}  

For each complex line in the normal space of a K\"ahler submanifold, this gives us a set of tangent complex lines:  

\begin{definition}[Holomorphic Principal Directions] 
\label{hpd}

{\em Let $M$ be a complex submanifold of a K\"ahler manifold $P$.  Let $\vec{u}$ be a normal vector to $M$ at $p$, and let $e_{1}$, $J(e_{1})$, $e_{2}$, $J(e_{2})$, $\cdots$, $e_{m}$, $J(e_{m})$ be an orthonormal basis of principal vectors for $A_{\vec{u}}$, with principal curvatures $\kappa_{1}$, $-\kappa_{1}$, $\kappa_{2}$, $-\kappa_{2}$, $\cdots$, $\kappa_{m}$, $-\kappa_{m}$ as in Proposition \ref{eqop}.} \\

{\em We will refer to the complex lines $Span_{\C}( e_{i} )$ in the tangent space $T_{p}M $as holomorphic principal directions for the complex line $Span_{\C}(\vec{u})$ in the normal space $\nu_{p}M$.} 
\end{definition}  

Now let $M$ be a complex submanifold of $\C P^{N}$.  We let $\nu^{Proj}M$ denote the projectivized normal bundle of $M$ in $\C P^{N}$ - in other words, the quotient of the unit normal bundle $\nu^{1}M$ by the Hopf action on each fibre.  For a complex projective manifold of complex dimension $m$, the fibres of $\nu^{Proj}M$ are therefore isometric to $\C P^{N-m-1}$.  Proposition \ref{IIrelationship} allows us to express the total absolute curvature of a complex projective manifold as an integral over $\nu^{Proj}M$:   

\begin{proposition} 
\label{int_over_nu_proj}

Let $M$ be a compact complex manifold of complex dimension $m$ holomorphically immersed in $\C P^{N}$, and let $\nu^{Proj}M$ be its projectivized normal bundle.  For each complex line in the normal space to $M$ at a point $p$, as in Definition \ref{hpd}, let $\kappa_{1}^{2}$, $\kappa_{2}^{2}$, $\cdots$, $\kappa_{m}^{2}$ be the squares of the principal curvatures for a set of holomorphic principal directions, and let $\sigma_{i}( \kappa^{2} )$ be the $i^{th}$ symmetric function of the $\kappa_{i}^{2}$, as in (\ref{formulaforcptc}).  Then: 

\begin{center} 

$\mathcal{T}(M) = \frac{4\pi}{Vol(\C P^{N})} \displaystyle \text{\LARGE $\int\limits_{\text{\tiny $\nu^{Proj}M$}} \int\limits_{\text{\tiny $0$}}^{\text{\tiny $\frac{\pi}{2}$}}$} \text{\LARGE $\vert$} \prod\limits_{i=1}^{m} \text{\large $($}\cos^{2}(r) - \kappa_{i}^{2} \sin^{2}(r) \text{\large $)$} \text{\LARGE $\vert$} \cos(r) \sin^{(2N-2m-1)}(r) \text{ dr $dVol_{\text{\tiny $\nu^{Proj}M$}}$}$

\medskip  

$= \frac{4\pi}{Vol(\C P^{N})} \displaystyle \text{\LARGE $\int\limits_{\text{\tiny $\nu^{Proj}M$}} \int\limits_{\text{\tiny $0$}}^{\text{\tiny $\frac{\pi}{2}$}}$} \text{\LARGE $\vert$} \sum\limits_{i=0}^{m}(-1)^{i}\sin^{(2N-2m-1+2i)}(r)\cos^{(2m-2i+1)}(r)\sigma_{i}(\kappa^{2})  \text{\LARGE $\vert$}  \text{ dr $dVol_{\text{\tiny $\nu^{Proj}M$}}$}.$

\end{center}
\end{proposition}

The factor $4\pi$ in Proposition \ref{int_over_nu_proj}, instead of $2$ as in Definition \ref{cptc}, is because the fibres of $\nu^{1}_{p}M$ over $\nu^{Proj}_{p}M$ have length $2\pi$.  For a complex projective hypersurface $M$, this allows us to express the total absolute curvature as an integral over $M$ itself:   

\begin{theorem}  
\label{hypatc}

Let $M$ be a compact complex manifold, of complex dimension $m$, holomorphically immersed in $\C P^{m+1}$.  For each $p$ in $M$, let $K_{1}$, $K_{2}$, $\cdots$, $K_{m}$ be the holomorphic sectional curvatures of a family of holomorphic principal directions at $p$.  Let $\sigma_{i}(K)$ represent the $i^{th}$ elementary symmetric function of the $K_{i}$. \\  

Then $K_{i} = 4 - 2\kappa_{i}^{2}$, and the total absolute curvature of $M$ is: \\  

$\mathcal{T}(M) =  \frac{4\pi}{Vol(\C P^{N})} \displaystyle \text{\LARGE $\int\limits_{\text{\tiny $M$}} \int\limits_{\text{\small $0$}}^{\text{\small $\frac{\pi}{2}$}}$} \text{\LARGE $\vert$}  \sum\limits_{i=0}^{m} \left( 1-3\sin^{2}(r) \right)^{(m-i)} \text{\footnotesize $\left( \frac{\sin^{2i}(r)}{2^{i}} \right)$} \sigma_{i}(K) \text{\LARGE $\vert$} \cos(r) \sin^{(2N-2m-1)}(r) \text{ dr $dVol_{M}$}.$

\end{theorem}

\begin{proof}

By the Gauss formula, the holomorphic sectional curvature $K_{i}$ of a holomorphic principal direction $Span_{\C}( e_{i} )$ is equal to $4 - 2\kappa_{i}^{2}$, where $\kappa_{i}, -\kappa_{i}$ are the principal curvatures of the principal vectors $e_{i}, J(e_{i})$.  Substituting this in the expression for $\mathcal{T}(M)$ in Proposition \ref{int_over_nu_proj}, we have: \\  

$\mathcal{T}(M) = \frac{4\pi}{Vol(\C P^{N})} \displaystyle \text{\LARGE $\int\limits_{\text{\tiny $M$}} \int\limits_{\text{\small $0$}}^{\text{\small $\frac{\pi}{2}$}}$} \text{\LARGE $\vert$} \prod\limits_{i=1}^{m} \left( \cos^{2}(r) - \text{\footnotesize $\left( \frac{4 - K_{i}}{2} \right)$} \sin^{2}(r) \right) \text{\LARGE $\vert$} \cos(r) \sin^{(2N-2m-1)}(r)  \text{\em dr $ dVol_{M}$}  $

\bigskip  

$= \frac{4\pi}{Vol(\C P^{N})} \displaystyle \text{\LARGE $\int\limits_{\text{\tiny $M$}} \int\limits_{\text{\small $0$}}^{\text{\small $\frac{\pi}{2}$}}$} \text{\LARGE $\vert$} \prod\limits_{i=1}^{m} \left( 1 - 3\sin^{2}(r) + K_{i} \text{\footnotesize $\left( \frac{\sin^{2}(r)}{2} \right)$} \right) \text{\LARGE $\vert$} \cos(r) \sin^{(2N-2m-1)}(r) \text{\em dr $ dVol_{M}$}  $

\bigskip  

$= \frac{4\pi}{Vol(\C P^{N})} \displaystyle \text{\LARGE $\int\limits_{\text{\tiny $M$}} \int\limits_{\text{\small $0$}}^{\text{\small $\frac{\pi}{2}$}}$} \text{\LARGE $\vert$} \sum\limits_{i=0}^{m} \left( 1-3\sin^{2}(r) \right)^{(m-i)} \text{\footnotesize $\left(\frac{\sin^{2i}(r)}{2^{i}} \right)$} \sigma_{i}(K) \text{\LARGE $\vert$} \cos(r) \sin^{(2N-2m-1)}(r)  \text{\em dr $ dVol_{M}$}.  $

\end{proof}

Calabi proved in \cite{Ca} that if a K\"ahler manifold admits a holomorphic isometric immersion into a complex space form, even locally, then this immersion is essentially unique:  

\begin{theorem}[Calabi, \cite{Ca}]  

Let $(M,h)$ be a K\"ahler manifold of complex dimension $m$, and let $U$ be a neighborhood in $M$ which admits a holomorphic isometric immersion into a complex space form $F(N,\lambda)$, of complex dimension $N$ and holomorphic sectional curvature $4\lambda$.  Suppose that the image of this immersion does not lie in any proper linear subspace of $F(N, \lambda)$. \\ 

Then the dimension $N$ of the ambient space is uniquely determined by the holomorphic sectional curvature $4\lambda$ and the metric $h$ on $U$, and the immersion is uniquely determined up to a holomorphic isometry of $F(N,\lambda)$.  

\end{theorem}

\begin{corollary} All holomorphic isometric immersions of a K\"ahler manifold (with a fixed metric) into complex projective space have the same total absolute curvature.  

\end{corollary}

In general, Theorem \ref{hypatc} does not give a completely intrinsic representation of $\mathcal{T}(M)$ because it does not characterize the holomorphic principal directions of $M$ in $\C P^{N}$ intrinsically.  However it follows from Calabi's theorem that the total absolute curvature of a complex projective manifold is actually part of its intrinsic geometry, and for curves $\Sigma$ in $\C P^{2}$, Theorem \ref{hypatc} does give a completely intrinsic representation of $\mathcal{T}(\Sigma)$: \\  

\textbf{Theorem \ref{curveformula}} {\em Let $\Sigma$ be a smooth curve in $\C P^{2}$, with $K$ the sectional curvature of its projectively induced metric.  Then: } \\ 

\begin{center}
$\mathcal{T}(\Sigma) = \displaystyle \frac{1}{\pi} \text{\LARGE $\int\limits_{\text{\small $\Sigma$}}$} \frac{(K-4)^{2} + 4}{6 - K} dA_{\Sigma}.$  
\end{center}

Recall that the sectional curvature of $\Sigma$ is bounded above by $4$, away from the value $6$ at which the integrand would be undefined. 

\begin{proof}[Proof of Theorem \ref{curveformula}]  

By Proposition \ref{hypatc}, the pointwise contribution to $\mathcal{T}(\Sigma)$ from a point of $\Sigma$ with curvature $K$ is:  

\begin{center}
\begin{equation}
\label{curvcalceqn}
\displaystyle \text{\LARGE $\int\limits_{0}^{\frac{\pi}{2}}$} \text{\Large $\vert$} \text{\small $1 + \left( \frac{K}{2} - 3 \right) \sin^{2}(r)$ } \text{\Large $\vert$} \cos(r) \sin(r) dr. 
\end{equation}
\end{center}
\medskip

When $r = 0$, $1 + (\frac{K}{2} - 3)\sin^{2}(r) = 1$.  When $r = \frac{\pi}{2}$, $1 + (\frac{K}{2} - 3)\sin^{2}(r) = \frac{K}{2} - 2$, which is less than or equal to zero for $K \leq 4$, with equality precisely when $K = 4$.  In general, $1 + (\frac{K}{2} - 3)\sin^{2}(r) = 0$ precisely when $\sin^{2}(r) = \frac{2}{6 - K}$, so when $r = \arcsin \scriptstyle (\sqrt{\frac{2}{6 - K}})$.  For $K \in (-\infty,4]$, $\scriptstyle \sqrt{\frac{2}{6 - K}}$ takes values in $(0,1]$, so $\arcsin \scriptstyle (\sqrt{\frac{2}{6 - K}})$ has a well-defined value in $(0, \frac{\pi}{2}]$.  Let $\alpha_{K}$ denote this value of $\arcsin \scriptstyle (\sqrt{\frac{2}{6 - K}})$.  We then evaluate the integral in (\ref{curvcalceqn}) over the intervals $[0, \alpha_{K}]$ and $[\alpha_{K}, \frac{\pi}{2}]$.  The resulting pointwise value for the total absolute curvature is:  \\  

\begin{center}

$\displaystyle \text{\LARGE $\int\limits_{0}^{\alpha_{K}}$} \left( 1 + \text{\footnotesize $\left(\frac{K}{2} - 3 \right)$} \sin^{2}(r) \right) \cos(r) \sin(r) dr - \text{\LARGE $\int\limits_{\alpha_{K}}^{\frac{\pi}{2}}$} \left( 1 + \text{\footnotesize $\left(\frac{K}{2} - 3 \right)$} \sin^{2}(r) \right) \cos(r) \sin(r) dr$

\bigskip  

$\displaystyle = (\frac{1}{2}) \left( \frac{K^{2} - 8K + 20}{24-4K} \right) = (\frac{1}{8}) \left( \frac{(K - 4)^{2} + 4}{6 - K} \right).$

\end{center}
\bigskip  

Substituting this in the expression for $\mathcal{T}(\Sigma)$ in Theorem \ref{hypatc}, together with the fact that $Vol(\C P^{2}) = \frac{\pi^{2}}{2}$, gives the result.

\end{proof}  

\begin{example}
\label{fermatconic}

The conic $\mathcal{F}$ described by $z_{0}^{2} + z_{1}^{2} + z_{2}^{2} = 0$ in $\C P^{2}$ is isometric to a round $2$-sphere with curvature $2$.  As a degree $2$ curve, $Area(\mathcal{F}) = 2 \times Area(\C P^{1}) = 2\pi$.  (More generally, as explained in the discussion after Proposition \ref{eqop}, Wirtinger's inequality implies that a degree $d$ curve in $\C P^{2}$ has area $d \times Area(\C P^{1}) = d\pi$.)  By Theorem \ref{curveformula}, we then have: 

\medskip
\begin{center}
$\mathcal{T}(\mathcal{F}) = \displaystyle \frac{1}{\pi} \text{\LARGE $\int\limits_{\text{\small $\mathcal{F}$}}$} \frac{(2 - 4)^{2} + 4}{6 - 2} dA_{\Sigma} = \frac{1}{\pi} \times 2\pi \times 2 = 4.$  
\end{center}
\medskip

This shows that Part B of Theorem \ref{cpcl} is optimal.    

\end{example}

More generally, for smooth degree $d$ curves in $\C P^{2}$, the following holds:   

\begin{proposition}
\label{curveestimate} 

Let $\Sigma_{d}$ be a smooth curve of degree $d$ in $\C P^{2}$. Then: 

\smallskip  

\begin{equation}
\label{curveesteqn}
2d^{2} - 4d + 4 \leq \mathcal{T}(\Sigma_{d}) \leq 2d^{2} \text{.}
\end{equation}

\end{proposition}

\smallskip

\begin{proof}  

$\Sigma_{d}$ is a compact Riemann surface of genus $\frac{(d-1)(d-2)}{2}$.  As explained in Example \ref{fermatconic}, its area is $d\pi$.  Let $K_{d}$ be the average sectional curvature of $\Sigma_{d}$.  By the Gauss-Bonnet formula, \\   

\begin{center}
$d\pi K_{d} = 2\pi \chi(\Sigma_{d}) = 2\pi(2 - (d-1)(d-2)) = 2\pi d(3 - d).$
\end{center}
\medskip  

Therefore, $K_{d} = 2(3 - d)$.  Because the function $\frac{(K - 4)^{2} + 4}{6-K}$ in Theorem \ref{curveformula} is convex on $(-\infty,4]$, Jensen's inequality then implies that: 
%Therefore, $K_{d} = 2(3 - d)$.  By applying Jensen's inequality to the function $\frac{(K - 4)^{2} + 4}{6-K}$ in Theorem \ref{curveformula}, which is convex on $(-\infty,4]$, we then have: \\  

\begin{center}
$\displaystyle Area(\Sigma_{d}) \left( \frac{(K_{d} - 4)^{2}  + 4}{6-K_{d}} \right) \leq \text{\LARGE $\int\limits_{\text{\small $\Sigma_{d}$}}$} \frac{(K - 4)^{2} + 4}{6-K} dA_{\Sigma_{d}} = \pi \times \mathcal{T}(\Sigma_{d}).$

\end{center} 

\medskip  

This implies that $2d^{2} - 4d + 4 \leq \mathcal{T}(\Sigma_{d})$. \\ 

On the other hand, $\frac{(K - 4)^{2} + 4}{6-K}$ is asymptotic, as $K \rightarrow -\infty$, to a line with slope $-1$.  It is straight-forward to check that the linear function $f(K) = 6 - K$ gives an upper bound for $\frac{(K - 4)^{2} + 4}{6-K}$ for all $K \leq 4$, with equality when $K = 4$, and that this is the best linear upper bound possible, in that a line with greater slope or smaller intercept will no longer give an upper bound for $\frac{(K - 4)^{2} + 4}{6-K}$ for all $K \in (-\infty, 4]$.  We then have: \\

\begin{center}

$\displaystyle \mathcal{T}(\Sigma_{d}) = \frac{1}{\pi} \text{\LARGE $\int\limits_{\text{\small $\Sigma_{d}$}}$} \frac{(K - 4)^{2} + 4}{6-K} dA_{\Sigma} \leq \frac{1}{\pi} \text{\LARGE $\int\limits_{\text{\small $\Sigma_{d}$}}$} (6 - K) dA_{\Sigma} = \frac{1}{\pi} \left( 6 Area(\Sigma_{d}) - 2\pi \chi(\Sigma_{d}) \right) = 2d^{2}.$

\end{center} 

\end{proof}  

Hulin proved in \cite{Hu} that if a K\"ahler-Einstein metric on a compact complex manifold is induced by a holomorphic embedding into complex projective space, then the metric has positive scalar curvature.  In particular, the only constant-curvature metrics on plane algebraic curves, with $K \equiv K_{d}$, occur on curves of degree $1$, and on curves of degree $2$ which are isometric (and thus congruent, by Calabi's theorem) to the curve in Example \ref{fermatconic}.  Other than for these curves, the lower bound for $\mathcal{T}(\Sigma_{d})$ in Proposition \ref{curveestimate} is strict.  The upper bound in Proposition \ref{curveestimate} is strict except for degree $1$ curves. \\  

A smooth degree $d$ curve in $\C P^{2}$ whose curvature is nearly constant will have total absolute curvature close to the lower bound $2d^{2} - 4d + 4$ in Proposition \ref{curveestimate}.  The equivalent statement for the upper bound $2d^{2}$ seems harder to formulate.  For example, it is shown by Vitter in \cite{Vi} that the degree $d$ Fermat curve $z_{0}^{d} + z_{1}^{d} + z_{2}^{d} = 0$, for $d \geq 3$, has $3d$ points at which its sectional curvature is maximal, equal to $4$.  This shows that for curves of arbitrarily high degree, one cannot replace the linear upper bound $6 - K$ in the proof of Proposition \ref{curveestimate} by a stronger upper bound, even though for any $\epsilon > 0$, a better linear upper bound is available for $K$ confined to the interval $(-\infty, 4 - \epsilon]$. \\ 

For a smooth curve in $\C P^{2}$, its degree determines its topology - in particular, its Betti numbers.  Proposition \ref{basicestimate} can then be translated into an inequality between the total curvature of a smooth plane curve and its degree: \\

\begin{corollary}[of Proposition \ref{basicestimate}]  Let $\Sigma_{d}$ be a smooth curve of degree $d$ curve in $\C P^{2}$.  Then: 
	\medskip 
	\begin{center}
		$2d^{2}-6d+6 \leq \mathcal{T}(\Sigma_{d})$. 
	\end{center}
\end{corollary}

We now have two results which give a lower bound for the total absolute curvature of a smooth plane curve in terms of its degree - the corollary of Proposition \ref{basicestimate} above and Proposition \ref{curveestimate}.  When $d = 1$, both results give the same estimate, that $\mathcal{T}(\C P^{1}) \geq 2$.   We know that in fact $\mathcal{T}(\C P^{1}) = 2$.  When $d = 2$, Proposition \ref{curveestimate} implies that $\mathcal{T}(\Sigma_{2}) \geq 4$, which also follows from Part B of Theorem \ref{cpcl}.  Proposition \ref{basicestimate} gives a weaker result, that $\mathcal{T}(\Sigma_{2}) \geq 2$.  Proposition \ref{curveestimate} also implies that $\mathcal{T}(\Sigma_{2}) < 8$.  When $d = 3$, Proposition \ref{basicestimate} implies that that $\mathcal{T}(\Sigma_{3}) \geq 6$.  Proposition \ref{curveestimate} implies that $10 < \mathcal{T}(\Sigma_{3}) < 18$. \\

For all $d \geq 3$, the lower bound in Proposition \ref{curveestimate} is stronger than the result in Proposition \ref{basicestimate}.  However, the upper and lower bounds in Proposition \ref{curveestimate} together show that Proposition \ref{basicestimate}, based ultimately on the original Chern-Lashof theorems, gives the right order of growth in the degree $d$ for the optimal lower bound for $\mathcal{T}(\Sigma_{d})$.  Proposition \ref{curveestimate} also shows that for all $d \geq 3$, the total absolute curvatures of degree $d$ curves are contained in an interval of length $4(d - 1)$, between $2d^{2} - 4d + 4$ and $2d^{2}$.  We note that the lower bound for degree $d+1$ curves, $2(d+1)^{2} - 4(d+1) + 4$, exceeds the upper bound for degree $d$ curves, $2d^{2}$, by $2$.  This implies that the total absolute curvature of a smooth curve in $\C P^{2}$ determines its degree.  We record this in the following: \\ 

\textbf{Proposition \ref{atcdeterminesdegree}} {\em Let $\Sigma$ be a smooth curve in $\C P^{2}$.  Then the degree of $\Sigma$ is the unique natural number $d$ such that $2d^{2} - 4d + 4 \leq \mathcal{T}(\Sigma) \leq 2d^{2}$.} \\  

By the Gauss-Bonnet formula, the (non-absolute) total curvature of a smooth plane curve determines its topology, and hence its degree for curves of degrees $3$ and greater.  However, the Gauss-Bonnet integral is the same for curves of degrees $1$ and $2$.  Moreover, the Gauss-Bonnet integral is the same for all smooth curves of a fixed topological type - it cannot distinguish between non-isomorphic curves with the same topology, or between geometrically distinct embeddings of isomorphic curves.  Total absolute curvature can distinguish these things in some cases - for example, smooth conics which are not congruent to the curve in Example \ref{fermatconic} will have total absolute curvature greater than $4$.  In this sense, total absolute curvature is a somewhat stronger invariant than total curvature. \\ 

Proposition \ref{atcdeterminesdegree} implies that the total absolute curvatures of complex submanifolds of $\C P^{2}$ belong to disjoint intervals in $\R_{\geq 2}$, with each interval associated to curves of a fixed degree.  Smooth plane curves of degree $d$, together with their embeddings, belong to a continuous and connected family (parametrized by an open, dense and connected subset of $\C P^{\text{\scriptsize $\binom{d+2}{2}$} - 1}$) so the total absolute curvatures of smooth degree $d$ curves in $\C P^{2}$ form a connected interval in $\R_{\geq 2}$.  By Theorem \ref{curveestimate} and Example \ref{fermatconic}, this interval is a point for $d=1$, is a sub-interval of $[4, 8)$ which includes the endpoint $4$ for $d = 2$, and is a sub-interval of $(2d^{2} - 4d + 4, 2d^{2})$ for $d \geq 3$.  Determining these intervals precisely, describing total absolute curvature as a function on the parameter space of smooth degree $d$ curves and studying the total absolute curvature of singular curves would strengthen these results. \\ 

One can fix a higher-dimensional complex projective space and seek a characterization of the total absolute curvatures of its closed complex submanifolds, along the lines of the results we have sketched above for complex submanifolds of $\C P^{2}$.  However, one can also approach this question from a different point of view: one can fix a compact complex manifold $M$, of projective type, and ask for the total absolute curvatures of all metrics on $M$ which can be induced by holomorphic immersions into complex projective space. 

%%%%%%%%%%%%%%%%%%%%%%%%%%%%%%%%%%%%%%%%%%%%%%%%%%%%%%%%%%%%%%%%%%%%%%%%%%%%%%%%%%%%%%%%%%%%%%%%%%%%%%%%%%%%%%%%%%%%%%%%%

\end{document}